\documentclass[12pt]{article}
\usepackage{amsmath,amsthm,amssymb,euscript,verbatim}
\newtheorem{theorem}{Theorem}[section]

\newtheorem{corollary}{Corollary}[section]
\newtheorem{remark}{Remark}[section]

\newtheorem{conjecture}{Conjecture}[section]
\newtheorem{definition}{Definition}[section]
\begin{document}
\title{\bf Taut and Dupin Submanifolds (Updated Version)}
\author{Thomas E. Cecil\thanks{Research supported by NSF
Grant No. DMS-9504535.}}
\maketitle
\begin{abstract}
This is an updated version of the paper \cite{CecMSRI} by the author which originally appeared in 1997.
The original paper was a survey of the closely related fields of taut and
Dupin submanifolds of Euclidean space, and this updated version includes many results in the field that have 
appeared since the publication of the original version. 
The emphasis is on stating
results in their proper context and noting areas for future research, and
relatively few proofs are given. The important class of isoparametric
hypersurfaces is surveyed in detail, as is the relationship between
the two concepts of taut and Dupin.  Also included is a brief introduction to
submanifold theory in Lie sphere geometry which is needed to state
many known results on Dupin submanifolds accurately.  The paper
concludes with detailed descriptions of the main known classification
results for both Dupin and taut submanifolds.  
\end{abstract}

\noindent
In a book published in 1822, Dupin \cite{D} determined which surfaces $M$
embedded in Euclidean 3-space ${\bf R}^3$ can be obtained as the
envelope of the family of spheres tangent to three fixed spheres.
These surfaces, known as the {\em cyclides of Dupin}, can all
be constructed by inverting a torus of revolution, a circular
cylinder, or a circular cone in a metric sphere  in ${\bf R}^3$ .  The cyclides of Dupin
were studied extensively in the nineteenth century (see, for example,
Cayley \cite{Cay}, Liouville \cite{Lio}, Maxwell \cite{Max}), and see Cecil-Ryan \cite[pp. 151--166]{CR7}
for a survey.
The cyclides of Dupin have several other important characterizations.  They
are the only surfaces $M$ in  ${\bf R}^3$ whose focal set consists
of two curves, which must, in fact, be a pair of focal conics.
This is equivalent to requiring that $M$ has two distinct principal
curvatures at every point, each of which is constant along each of its
corresponding lines of curvature.  It is also equivalent to the
condition that all lines of curvature in both families are either arcs of circles
or segments of straight lines.

The cyclides of Dupin re-appeared in modern differential geometry in a paper
published in 1970 by Banchoff \cite{Ban1}.  He considered 
compact surfaces
$M$ embedded in ${\bf R}^3$ with the property that every metric
sphere in ${\bf R}^3$ cuts $M$ into at most two pieces, i.e.,
the {\em spherical two-piece property (STPP)}.  For surfaces,
the {\em STPP} 
is equivalent to requiring that $M$ be {\em taut}, i.e., every
non-degenerate Euclidean distance function 
$L_p(x) = | p-x |^2$,
$p \in  {\bf R}^3$,
has the minimum number of critical points required by
the Morse inequalities on $M$.  Banchoff showed that tautness implies
that $M$ must be a metric sphere or a cyclide of Dupin in ${\bf R}^3$, and the
close link between the notions of taut and Dupin submanifolds was established.

Following the terminology of Pinkall \cite{P2}, a hypersurface $M$ in  ${\bf R}^n$ is said to be {\em Dupin}
if along each curvature surface, the corresponding principal
curvature is constant.  A Dupin hypersurface $M$
is called {\em proper Dupin}
if the number of distinct principal curvatures is constant on $M$.
These concepts can both be generalized in a natural way to submanifolds
of codimension greater than one in  ${\bf R}^n$.  A fundamental
result in the theory due to Pinkall \cite{P4} (and independently to Miyaoka \cite{Mi8}
in the case of codimension one) is that a taut submanifold
must be Dupin, but not necessarily proper Dupin.  Conversely, the work of
Thorbergsson \cite{Th1} and Pinkall \cite{P4} shows
that a compact proper Dupin submanifold embedded in  ${\bf R}^n$
must be taut. 

This paper is a survey of the major results on taut and Dupin
submanifolds.  We concentrate on stating the results in their
proper context and noting areas for future research and give
very few proofs.  In particular, we do not repeat proofs
of many fundamental results in the field which can be found
in the books \cite{CR7}, \cite{CR8}, co-authored by the author and Patrick J. Ryan.   
We will not attempt to cover the
related field of tight immersions, and the reader is referred to the
paper by Banchoff and K\"{u}hnel  \cite{BK} for a survey of that field.

Isoparametric submanifolds of Euclidean space ${\bf R}^n$ or the unit sphere $S^n \subset {\bf R}^{n+1}$ are both taut and Dupin, and they play an important role in the theory.  These will be reviewed in Section \ref{sec:isoparametric-submanifolds} of this paper, and the reader
is also referred to the survey articles of Terng \cite{Te5},
Thorbergsson \cite{Th6}--\cite{Th7}, Cecil \cite{Cec10}, Chi \cite{Chi-survey}, and the book
 \cite[pp. 85--184]{CR8} of Cecil and Ryan. 

There is also an
extensive amount of research on real hypersurfaces
with constant principal curvatures in complex
space forms, which is
surveyed in  the paper by Niebergall and Ryan \cite{NieR}, and in the book \cite[pp. 343--551]{CR8} of Cecil and Ryan.

We now give a brief overview 
of the contents of this paper.
In Section \ref{sec:taut-submanifolds}, we review the critical point theory and
submanifold theory needed to formulate the definition of
a taut submanifold, and we list some basic results and methods
for constructing taut embeddings.  In Section \ref{sec:isoparametric-submanifolds} we list the primary 
known results for isoparametric hypersurfaces in spheres, which
play an important role in the theory of taut and Dupin hypersurfaces.

In Section \ref{sec:dupin-submanifolds}, we give the definition of a Dupin submanifold and review
Pinkall's standard local constructions of proper Dupin
hypersurfaces with an arbitrary number of distinct principal
curvatures and respective multiplicities.  In Section \ref{sec:relationship-taut-dupin}, we discuss
the relationship between the taut and Dupin conditions in
detail.

Many of the main classifications of proper Dupin submanifolds
are done in the context of Lie sphere geometry. In Section \ref{sec:submanifolds-lie-geometry},
we give a brief introduction to this theory in order
to be able to explain these classifications accurately.

Section \ref{sec:compact-proper-dupin} is a survey of the known results on compact, connected proper
Dupin hypersurfaces.  Thorbergsson \cite{Th1} applied the
work of M\"{u}nzner \cite{Mu}--\cite{Mu2} to show that
the number $g$ of distinct principal curvatures of such a
hypersurface must be $1,2,3,4$ or 6, the same as for an isoparametric
hypersurface in a sphere.  
For some time, it was 
conjectured that every compact, connected proper 
Dupin hypersurface is equivalent by a Lie
sphere transformation to an isoparametric hypersurface.
However, this is not the case, as examples constructed
by Pinkall and Thorbergsson \cite{PT}, and by Miyaoka and
Ozawa \cite{MO}, demonstrate.  We describe these examples in
detail.

In Section \ref{sec:local-dupin}, we study the local classifications of proper Dupin
hypersurfaces which have been obtained using Lie sphere
geometry.  We describe the known results and mention 
areas for further research.

Finally, in Section \ref{sec:classifications-taut-submanifolds}, we survey the known classifications of taut
embeddings.
To some extent, this section can be
read independently from the rest of the paper, although some references
to the previous sections are necessary.

I  would like to thank my collaborators, Quo-Shin Chi, Gary Jensen, and Patrick Ryan, for their many insights on this material over the years.
I would also like to thank Thomas Banchoff, Peter Breuer, S.-S. Chern,  
Jianquan Ge, Wolfgang K\"{u}hnel,  Nicolaas Kuiper,
Martin Magid, Katsumi Nomizu,
Chuu-Lian Terng, 
and Gudlaugur Thorbergsson for their comments on earlier
versions of this paper.  I also wish to acknowledge several helpful comments by the referee.

\section{Taut Submanifolds}
\label{sec:taut-submanifolds}

We begin with a brief review of the critical point theory
and submanifold theory needed
to formulate the definition of tautness.
In this paper, all manifolds are assumed to be  connected unless
explicitly stated otherwise.  Let $M$ be a smooth,
connected $n$-dimensional manifold, and let $\phi$ be a smooth
real-valued function defined on $M$.  A point $x \in M$ is a
{\em critical point} of $\phi$ if the differential $\phi_*$ is
zero at $x$.  The critical point $x$ is {\em non-degenerate}
if the Hessian $H$ of $\phi$ is a non-degenerate bilinear form
at $x$, and otherwise it is said to be {\em degenerate}.  The
{\em index} of a non-degenerate critical point $x$ is equal to
the index of $H$ as a bilinear form, i.e., the dimension of 
a maximal subspace on which $H$ is negative-definite.  The
function $\phi$ is called a {\em Morse function} or
{\em non-degenerate function}, if it has only non-degenerate 
critical points on $M$.

Let $\phi$ be a Morse function on $M$ such that the set
\begin{displaymath}
M_r(\phi) = \{x \in M \mid \phi(x) \leq r \}
\end{displaymath}
is compact for all $r \in {\bf R}$.  Of course, this is true
for any Morse function on a compact manifold $M$.  Let
$\mu_k(\phi,r)$ be the number of critical points of $\phi$
of index $k$ on $M_r(\phi)$.  If $M$ is compact, let
$\mu_k(\phi)$ denote the number of critical points of index
$k$ on $M$.  For a field ${\bf F}$, let
\begin{displaymath}
\beta_k(\phi,r,{\bf F}) = {\rm dim}_{\bf F} H_k(M_r(\phi);{\bf F})
\end{displaymath}
be the $k$-th ${\bf F}$-Betti number of $M_r(\phi)$, and let
$\beta_k(M;{\bf F})$ be the $k$-th ${\bf F}$-Betti number of a compact $M$.
Then the {\em Morse inequalities} (see, for example, Morse-Cairns
\cite[p. 270]{MC}) state that
\begin{displaymath}
\mu_k(\phi,r) \geq \beta_k(\phi,r,{\bf F})
\end{displaymath}
for all ${\bf F}, k, r,$ and for a compact $M$,
\begin{displaymath}
\mu_k(\phi) \geq \beta_k(M;{\bf F})
\end{displaymath}
for all ${\bf F}, k.$  A Morse function $\phi$ on $M$ is said to be
{\em perfect} if $\phi$ has the minimum number of critical points
possible by the Morse inequalities, i.e., if there exists a field ${\bf F}$
such that
\begin{equation}
\label{eq:perfect-fn}
\mu_k(\phi,r) = \beta_k(\phi,r,{\bf F})
\end{equation}
for all $k, r.$  For a compact manifold $M$, this is equivalent to 
the condition $\mu_k(\phi) = \beta_k(M;{\bf F})$ for all $k$.  Equivalently,
it can be shown (see, for example, Morse-Cairns \cite[p. 260]{MC}) that a
Morse function $\phi$ on a compact manifold $M$ is perfect if there
exists a field ${\bf F}$ such that for all $k,r,$ the map on homology
\begin{equation}
\label{eq:perfect-compact-M}
H_k(M_r(\phi);{\bf F}) \rightarrow H_k(M;{\bf F})
\end{equation}
induced by the inclusion of $M_r(\phi)$ in $M$ is injective.
This formulation has proven to be quite useful in the theory of
tight and taut immersions.

Let $f:M \rightarrow {\bf R}^n$ be a smooth immersion of a manifold $M$
into $n$-dimensional Euclidean space.  Since $f$ is an immersion,
it is an embedding on a suitably small neighborhood of any point
$x \in M$.  Thus, for local calculations, we often identify the
tangent space $T_xM$ with its image $f_*(T_xM)$ under the
differential $f_*$ of $f$.  Suppose that $X \in T_xM$ and $\xi$
is a field of unit normal vectors to $f(M)$ defined on a
neighborhood of $x$.  Then we have the fundamental equation,
\begin{equation}
\label{eq:fundamental-eq}
D_X\xi = -A_{\xi}X + \nabla_X^{\perp}\xi,
\end{equation}
where $-A_{\xi}X$ is the component of $D_X\xi$ tangent to $M$,
and $\nabla_X^{\perp}\xi$ is the component normal to $M$.
Here $A_{\xi}$ is a symmetric tensor of type (1,1) on $M$ called the
{\em shape operator} determined by $\xi$, and $\nabla^{\perp}$ is
a covariant derivative operator 
in the normal bundle of $M$ called the {\em normal connection}.
The eigenvalues of $A_{\xi}$ are called the {\em principal
curvatures} of $A_{\xi}$.  When $f(M)$ is a hypersurface, a local
field of unit normal vectors $\xi$ is determined up to a sign.
In this case, the shape operator $A_{\xi}$ is often denoted
simply by $A$, and the eigenvalues of $A_{\xi}$ are determined
up to a sign, depending on the choice of $\xi$.  In that case,
these eigenvalues are called the principal curvatures of
$M$ or of $f$.

The {\em normal exponential map} $F$ from the normal bundle
$N(M)$ to ${\bf R}^n$ is defined by
\begin{displaymath}
F(x,\eta) = f(x) + \eta,
\end{displaymath}
where $\eta$ is a normal vector to $f(M)$ at $f(x)$.  A point
$p \in {\bf R}^n$ is called a {\em focal point of multiplicity}
$m$ {\em of} $(M,x)$, if 
$p=F(x,\xi)$ and the differential $F_*$ has nullity $m$
at $(x,\xi)$.  A point $p \in {\bf R}^n$ is called a focal
point of $M$ if $p$ is a focal point of $(M,x)$ for some
$x \in M$.  The set of all focal points of $M$ is called the
{\em focal set} of $M$.  Since $N(M)$ and  ${\bf R}^n$ have the same
dimension, Sard's Theorem implies that the focal set of $M$
has measure zero in  ${\bf R}^n$.  A direct computation
(see, for example, \cite[p. 34]{Mil}) shows that if
$p = F(x,t\xi)$, where $|\xi |= 1$, then $p$
is a focal point of $(M,x)$ of multiplicity $m$ if and only if
$1/t$ is a principal curvature of $A_{\xi}$ of multiplicity $m$.
In this paper, we often consider an immersion
$f:M \rightarrow S^n$ into the unit sphere $S^n$ in ${\bf R}^{n+1}$.
In that case, one can also define the notions of normal
exponential map and focal point in a manner analogous to the
definitions given here for submanifolds of Euclidean space.
 
A {\em Euclidean distance function} is a function 
$L_p:{\bf R}^n \rightarrow {\bf R}$ given by the formula, 
$L_p(q) = | p-q |^2$,
where $p \in {\bf R}^n$.
The level sets of $L_p$ are spheres centered at the point $p$.  Let
$f$ be an immersion of a smooth manifold $M$ into ${\bf R}^n$.
We consider the restriction of $L_p$ to $M$ defined by
$L_p(x) =  | p-f(x)|^2$.  It is well-known (see, for example,
Milnor \cite[pp. 33-38]{Mil}) 
that $L_p$ has a critical point at $x \in M$ if
and only if $p$ lies along a normal line to $f(M)$ at $f(x)$.
The critical point $x$ is degenerate precisely when $p$ is a
focal point of $(M,x)$.  If $L_p$ has a non-degenerate critical point
at $x$, then its index is the number of focal points of $(M,x)$
(counting multiplicities) on the line segment from $p$ to
$f(x)$.  Since the set of focal points of $f$ has measure zero
in ${\bf R}^n$, $L_p$ is a Morse function for almost all
$p \in {\bf R}^n$.  The immersion $f$ is said to be {\em taut}
if every Morse function of the form $L_p$ is perfect, i.e.,
there exists a field ${\bf F}$ such that
\begin{displaymath}
\mu_k(L_p,r) = \beta_k(M_r(L_p);{\bf F})
\end{displaymath}
for every Morse function of the form $L_p$ and for every $k,r$.
This definition makes sense for non-compact manifolds, and there do
exist taut immersions of non-compact manifolds, e.g., a circular
cylinder in ${\bf R}^3$, but most results deal with compact
manifolds.

As in the theory of tight immersions, there is a formulation of
tautness due to Kuiper in terms of \v{C}ech 
homology, which has proven to be 
very useful in establishing certain fundamental results (see the books 
of Cecil-Ryan \cite[pp. 113--125]{CR7}, \cite[pp. 40--45]{CR8} for more detail).  So far 
the field $F={\bf Z}_2$ has been sufficient for almost all
considerations, so we will use it almost exclusively here.  Recall that
a map $f:M \rightarrow {\bf R}^n$ is said to be {\em proper}
if $f^{-1}K$ is compact for every compact subset $K$ of
${\bf R}^n$.  Using \eqref{eq:perfect-compact-M} and the \v{C}ech theory, one can show that a
proper immersion $f$ of a manifold $M$ into ${\bf R}^n$ is taut
if and only if for every closed ball $B$ in ${\bf R}^n$, the
induced homomorphism
\begin{equation}
\label{eq:def-taut}
H_i(f^{-1}B) \rightarrow H_i(M)
\end{equation}
in \v{C}ech homology with ${\bf Z}_2$-coefficients is injective for
every $i$.    
The use of \v{C}ech homology allows one to use all closed
balls in ${\bf R}^n$ rather than only those determined by level
sets of non-degenerate distance functions.  Note that this formulation
of tautness makes sense even if $f$ is only assumed to be a
proper continuous map and $M$ a topological space.  In that case,
$f$ is called a {\em taut map}.

A few key facts which follow quickly from the definition are
the following.  First, a taut immersion must be an embedding.
In fact, this is true even if $f$ is only assumed to be 
{\em 0-taut}, i.e., the induced homomorphism \eqref{eq:def-taut} is injective
for $i=0$.  For a compact manifold $M$, 0-tautness is
equivalent to the {\em spherical two-piece property (STPP)}
of Banchoff \cite{Ban1} which requires that $f^{-1}\Omega$ is connected
for every closed ball, complement of an open ball or closed
half-space $\Omega$.  The {\em STPP} is also equivalent 
to the condition that every
Morse function of the form $L_p$ has exactly one local
maximum and one local minimum. This condition is equivalent to tautness
for compact manifolds of dimension 2 by the Morse inequalities.  

The {\em STPP} is quite strong,
and for a while, every known {\em STPP} embedding was actually
taut, but Curtin \cite{Cur} showed that there exist {\em STPP} 
embeddings that are not taut.  Specifically,
an embedding $f:M \rightarrow {\bf R}^n$
is said to be $k$-{\em taut} if the induced homomorphism \eqref{eq:def-taut}
is injective  for all $i \leq k$.  Curtin found substantial
embeddings of $S^n$ into $S^{n+d}, d \geq 1$, which are $k$-taut
but not $(k+1)$-taut for every $n \geq 3$ and every
$k \geq 0$ provided that $(d+1)(k+2)\leq n+1$.  He also produced
$k$-taut embeddings of manifolds other than spheres.  In a second
paper, Curtin \cite{Cur1} introduced a notion of tautness
for manifolds with boundary. 

As noted above, tautness can be studied for maps defined on spaces
which are not manifolds.  In fact, in his first paper on the {\em STPP},
Banchoff \cite{Ban1} determined all {\em STPP} subsets of the plane.
Later,
Kuiper \cite{Ku2} determined all taut subsets of ${\bf R}^2$
and all compact taut {\em ANR} subsets of ${\bf R}^3$. 

Next, a taut embedding $f$ of a compact manifold $M$ into
${\bf R}^n$ must be {\em tight}, i.e., every non-degenerate
linear height function $l_p(x) = \langle p,f(x) \rangle,$ for $p$
a unit vector in ${\bf R}^n$, is perfect.  (Here $\langle \ ,\ \rangle$ denotes the Euclidean inner product on ${\bf R}^n$.)
This is easily
shown using \v{C}ech homology, since a closed half-space can
be obtained as the limit of closed balls (see, for example, \cite[pp. 53--55]{CR8}). 
A map $f$ of a
compact topological space into ${\bf R}^n$ is said to have
the {\em two-piece property (TPP)} if $f^{-1}h$ is connected
for every closed half-space $h$, and of course, the {\em STPP} implies
the {\em TPP}. 
As it turns out,
tautness is a much stronger condition than tightness.

It is sometimes said that tautness is equivalent to the
combination of tight and spherical.  This is true in the following
sense.  First, suppose that $f$ is an embedding of a compact
manifold $M$ into ${\bf R}^{n+1}$ which lies in the unit
sphere $S^n$ in  ${\bf R}^{n+1}$, in which case we say that
$f$ is {\em spherical}.  Then if $f$ is a tight immersion into
${\bf R}^{n+1}$, it must also be taut, because the intersection
of any closed ball $B$ with $S^n$ can be realized as the
intersection of a closed half-space with $S^n$.
Note also that the distance in $S^n$ from $p$ to $f(x)$ is given
by the {\em spherical distance function} 
$d_p(x) = \cos^{-1} (l_p(x))$, which has the same critical points
as $l_p$.  Thus, for simplicity, we usually use linear height functions
rather than spherical distance functions in treating taut
submanifolds of $S^n$.
Next, if
$P_q:S^n - \{q\} \rightarrow {\bf R}^n$ is stereographic
projection with pole $q$ not in $f(M)$, then $P_q \circ f$
is a taut embedding of $M$ into  ${\bf R}^n$, since $P_q$ maps
a metric ball in $S^n$ to either a closed ball, the complement
of an open ball, or a closed half-space in  ${\bf R}^n$.  Thus,
$f$ is tight and spherical if and only if $P_q \circ f$
is taut. This was first observed by Banchoff \cite{Ban1}.

Hence, the theory of taut embeddings of compact 
manifolds is essentially the same whether one maps  into
${\bf R}^n$ or $S^n$, and we work with whichever ambient space
is more convenient for the problem at hand.
By the same type of reasoning, tautness is easily shown
to be invariant under {\em M\"{o}bius transformations}, i.e.,
conformal transformations of $S^n$ onto itself, since M\"{o}bius
transformations map hyperspheres to hyperspheres in $S^n$.

These considerations lead to a fundamental result on the bound
on the codimension of a taut embedding.  Recall that an
immersion $f:M \rightarrow {\bf R}^n$ is said to be
{\em substantial} if the image $f(M)$ does not lie in
any affine hyperplane in ${\bf R}^n$. 

Kuiper \cite{Ku1} showed that if
$f:M^k \rightarrow {\bf R}^n$ is a substantial {\em TPP}
immersion of a 
compact manifold $M^k$, then 
\begin{equation}
\label{eq:bound-on-codimension}
n \leq k(k+3)/2.
\end{equation} 
Furthermore, if equality holds in \eqref{eq:bound-on-codimension},
then $f$ must be a standard embedding 
of a real projective
space ${\bf RP}^k$ into a metric sphere $S^{n-1} \subset  {\bf R}^n$, up to a projective transformation of ${\bf R}^n$.
In the case where equality holds, the result is due to Kuiper  \cite{Ku1} for $k=2$, and
to Little and Pohl \cite{LP} for $k>2$ (see, for example, \cite[pp. 98--108]{CR7}).

\begin{remark}{Standard embeddings of projective spaces}
\label{rem:std-embeddingd-projective-spaces}

\noindent
{\rm The simplest case of a standard embedding of a projective space
is the well-known embedding of the real projective plane
${\bf RP}^2$ into a metric sphere $S^4 \subset {\bf R}^5$, which is sometimes called a {\em spherical Veronese
surface}, which we now describe.
Let $S^2$ be the unit sphere in ${\bf R}^3$ given by the equation
\begin{equation}
\label{eq:unit-sphere-r3}
x^2 + y^2 + z^2 = 1. 
\end{equation}
Define a map from $S^2$ into ${\bf R}^6$ by
\begin{equation}
\label{eq:Veronese-map-r6}
(x,y,z) \mapsto (x^2, y^2, z^2, \sqrt{2} yz, \sqrt{2} zx, \sqrt{2} xy).
\end{equation}
One can easily check that this map takes the same value on antipodal points of $S^2$, and so it induces a map $\phi: {\bf RP}^2 \rightarrow {\bf R}^6$.  An elementary calculation
then proves that $\phi$ is a smooth embedding of ${\bf RP}^2$.  Furthermore, if $(u_1,\ldots,u_6)$ are the standard coordinates on ${\bf R}^6$, then the image of $\phi$
lies in the Euclidean hyperplane ${\bf R}^5 \subset {\bf R}^6$ given by the equation,
\begin{equation}
\label{eq:Veronese-hyperplane}
u_1 + u_2 + u_3 = 1, 
\end{equation}
since $x^2 + y^2 + z^2 = 1$.  One can easily show further that $\phi$ is a substantial embedding into ${\bf R}^5$, and that the image of $\phi$ is contained in the unit
sphere $S^5 \subset {\bf R}^6$ given by the equation,
\begin{equation}
\label{eq:unit-sphere-r6}
u_1^2 + \cdots + u_6^2 = 1. 
\end{equation}
Thus, $\phi$ is a substantial embedding of ${\bf RP}^2$ into the 4-sphere $S^4 = S^5 \cap {\bf R}^5$.

The other standard embeddings of projective spaces are generalizations of the spherical Veronese surface.
For a detailed description of the standard embeddings of projective spaces (including projective spaces over the complex,
quaternion, and Cayley numbers), and a proof that they are taut,
see Tai \cite{Tai} or Cecil-Ryan \cite[pp. 74--83]{CR8}.}
\end{remark}

Since tautness implies tightness, Banchoff \cite{Ban1}
observed that the following basic result for compact manifolds follows from the work of Kuiper \cite{Ku1}, 
and Little-Pohl \cite{LP},
and then Carter and West \cite{CW1} extended it to non-compact
manifolds. (See also \cite[p. 124]{CR7}
for a proof.) Actually, in the compact case only the {\em STPP} 
is required,
and in the non-compact case, all that is required is that every
non-degenerate distance function $L_p$ have exactly one local
minimum and no local maxima on $M^k$.\\

\begin{theorem}
\label{thm:kuiper-bound-on-codimension}
Let $f:M^k \rightarrow {\bf R}^n$ be a substantial taut
embedding of a $k$-dimensional manifold.\\
{\rm (a)} If $M^k$ is compact, then $n \leq k(k+3)/2$. Furthermore, if $n = k(k+3)/2$,
then $f$ is a standard embedding of a real projective
space ${\bf RP}^k$.\\
{\rm (b)}   If $M^k$ is non-compact, then $n \leq k(k+3)/2 -1$.  Furthermore, if $n = k(k+3)/2 - 1$,
then $f(M^k)$ is the image under stereographic
projection of a standard embedding of a real projective space ${\bf RP}^k \subset S^n$, where the pole of the 
stereographic projection
is on the image of the standard embedding in $S^n$.
\end{theorem}

We next note that tautness is invariant under the process of taking parallel hypersurfaces
and tubes (see, for example, \cite[pp. 57--58]{CR8}).
Suppose that $f$ is a taut embedding of a compact 
$(n-1)$-dimensional manifold $M$ into $S^n$.  Then
$f(M)$ is orientable, so let $\xi$ be a field of unit
normal vectors to $f(M)$ in $S^n$.  The {\em parallel
hypersurface} to $f$ at signed distance $t$ is given by
the map $f_t:M \rightarrow S^n$,
\begin{equation}
\label{eq:f_t-spherical}
f_t(x) = \cos t \,  f(x) + \sin t \, \xi (x).
\end{equation}
Thus, $f_t(x)$ is obtained by traveling a signed distance $t$
along the normal geodesic to $f(M)$ through $f(x)$.  For
sufficiently small values of $t$, $f_t$ is also an
embedding of $M$, and $f$ is taut if and only if $f_t$ is
taut, since each linear height function $l_p$
has the same critical points on $f_t(M)$ as on $f(M)$.

This type of consideration is also valid for submanifolds
of codimension greater than one in $S^n$.  If
$\phi :V \rightarrow S^n$ is a compact submanifold of
codimension greater than one, we consider the tube $\phi_t$
of radius $t$ around $\phi(V)$ in $S^n$.  For sufficiently
small $t$, the map $\phi_t$ is an embedding of the unit
normal bundle $B^{n-1}$ of $\phi(V)$ into $S^n$.
Furthermore, $\phi$ is a taut  embedding of $V$ if and only if
$\phi_t$ is a taut embedding of $B^{n-1}$.  To see this, one
first computes that every non-degenerate height function
$l_p$ has twice as many critical points on the tube $\phi_t$ as it has
on $\phi$.  Since the sum of the ${\bf Z}_2$-Betti numbers
of $B^{n-1}$ is twice the sum of the ${\bf Z}_2$-Betti numbers
of $V$ (see Pinkall \cite{P4}), tautness is preserved.

\begin{remark}{Tautness is invariant under Lie sphere transformations}
\label{rem:tautness-Lie-invariant}

\noindent
{\rm More generally, tautness is invariant under Lie sphere transformations.  This was proven in a paper
by \'{A}lvarez Paiva \cite{Alv}, and we now briefly discuss his approach in that paper.
(See also \cite[pp. 82-95]{Cec9}, \cite[pp. 224--231]{CR8}, or the paper \cite{Cec-lie-taut}.)
To formulate this result properly, it is necessary to consider submanifold theory in the context of Lie sphere geometry, since Lie sphere transformations act on the space of oriented hyperspheres in $S^n$, and not on the space $S^n$ itself.  
(See Section \ref{sec:submanifolds-lie-geometry} for more detail.)

\'{A}lvarez Paiva \cite{Alv} first extended
the definition of tautness of submanifolds of $S^n$ to 
the concept of Lie-tautness of Legendre submanifolds of the contact manifold 
$\Lambda^{2n-1}$ of projective lines on the Lie quadric $Q^{n+1}$.

He then showed that this definition of Lie-tautness has the property that if $\phi:V \rightarrow S^n$ is an embedding of a compact, connected manifold $V$, then $\phi(V)$ is a taut submanifold of $S^n$ if and only if the Legendre lift $\lambda$
of $\phi$ is Lie-taut. Furthermore, he showed that
this notion of Lie-tautness is invariant under the action of Lie sphere transformations on Legendre submanifolds.  

As a consequence, one gets that
if $\phi:V \rightarrow S^n$ and $\psi:V \rightarrow S^n$ are two embeddings of a compact, connected  
manifold $V$ into $S^n$, 
such that their corresponding Legendre lifts are related by a Lie sphere transformation,
then $\phi$ is a taut embedding if and only if $\psi$ is a taut embedding.  Thus, in that sense, tautness is invariant under Lie sphere transformations.

The key idea in this approach of \'{A}lvarez Paiva \cite{Alv} is to formulate tautness in terms of real-valued
functions on $S^n$ whose level sets form a parabolic pencil 
of unoriented spheres in $S^n$, and then show that this is equivalent to the usual formulation 
of tautness in terms of spherical distance functions, whose level sets in $S^n$ form a pencil of unoriented concentric spheres.}
\end{remark}

\section{Isoparametric Submanifolds}
\label{sec:isoparametric-submanifolds}

Isoparametric submanifolds in ${\bf R}^n$ or in the sphere $S^n \subset {\bf R}^{n+1}$ are taut, and they play an important role in the theory of taut  and Dupin submanifolds.  We begin our treatment of this subject with a discussion of isoparametric hypersurfaces in real space forms, and later we
treat the case of codimension greater than one.

By a real {\em space form}
of dimension $n$, we mean a complete, connected, simply connected manifold $\widetilde{M}^n (c)$ with constant sectional curvature $c$.  If $c=0$, then $\widetilde{M}^n (c)$ is the $n$-dimensional Euclidean space ${\bf R}^n$; if
$c=1$, then $\widetilde{M}^n (c)$ is the unit sphere $S^n \subset {\bf R}^{n+1}$; and if $c= -1$, then $\widetilde{M}^n (c)$ is the $n$-dimensional real hyperbolic space $H^n$ (see, for example, \cite[Vol. I, pp. 204--209]{KN}).  For any value of $c>0$, this subject is basically the same as for the case $c=1$, and for any value of $c<0$, it is very similar to the case $c=-1$.  So we will restrict our attention to the cases $c = 0, 1, -1$.

The original definition of a family of isoparametric hypersurfaces in a real space
form $\widetilde{M}^{n+1}$ was formulated in terms of the level sets of an isoparametric function, as we now describe. 
Let $F:\widetilde{M}^{n+1} \rightarrow {\bf R}$ be a nonconstant smooth function. The classical Beltrami
differential parameters of $F$ are defined by
\begin{equation}
\label{eq:1-beltrami}
\Delta_1 F = |{\rm grad}\  F|^2, \quad \Delta_2 F = \Delta F\ ({\rm Laplacian}\ {\rm of}\  F),
\end{equation}
where grad $F$ denotes the gradient vector field of $F$.

The function $F$ is said to be {\em isoparametric} if there exist smooth functions $\phi_1$ and $\phi_2$ from ${\bf R}$ to 
${\bf R}$ such that
\begin{equation}
\label{eq:1-beltrami-isoparametric}
\Delta_1 F = \phi_1 \circ F, \quad \Delta_2 F = \phi_2 \circ F.
\end{equation}
That is, both of the Beltrami differential parameters are constant on each level set of $F$. This is the origin of the 
term {\em isoparametric}.
The collection of level sets of an isoparametric function is called an {\em isoparametric family} of sets in
$\widetilde{M}^{n+1}$.

An isoparametric family in ${\bf R}^{n+1}$ consists of either
parallel planes, concentric spheres, or coaxial spherical cylinders,
and their focal sets.
This was first shown for $n=2$ by Somigliana \cite{Som} 
(see also B. Segre \cite{Seg1} and Levi-Civita  \cite{Lev}), and for arbitrary $n$ by B. Segre \cite{Seg}.

In the late 1930's, shortly after the publication of the papers of 
Levi-Civita  \cite{Lev} and Segre \cite{Seg},  
Cartan \cite{Car1}--\cite{Car4} began a study of isoparametric families in arbitrary
real space forms. 
In Section 2 of his first paper on the subject,  Cartan \cite{Car1} showed that an isoparametric family, defined as the collection of level sets of an isoparametric function $F$, is 
locally equivalent to a family of parallel hypersurfaces, each of which has constant principal curvatures, together with their focal submanifolds.  

Conversely, let $f_t:M \rightarrow \widetilde{M}^{n+1}$, $-\epsilon < t < \epsilon$, 
be a family of parallel hypersurfaces such that $f_0$
has constant principal curvatures.  We see that this is an isoparametric family of hypersurfaces as follows.
By the well known formulas for the principal curvatures of a parallel hypersurface in a real space form (see, for example, \cite[pp. 17--18]{CR8}), each $f_t M$ has constant principal curvatures,
and thus each $f_t M$ has constant mean curvature.  Then the function $F$ defined by
$F(x) = t$, if $x \in f_t M$, is a smooth function defined on an open subset of $\widetilde{M}^{n+1}$ with the property that
${\rm grad}\  F = \xi$ is a unit length vector field such that $\widetilde{\nabla}_{\xi} \xi = 0$.

Furthermore,
since the function $\rho = |{\rm grad}\  F| = 1$ is constant, the constancy of the mean
curvature $h$ on each level hypersurface $f_t M$ implies that the Laplacian of $F$ is also constant
on each level set $f_t M$, and therefore $F$ is an isoparametric function, as defined by  equation 
(\ref{eq:1-beltrami-isoparametric}). Thus, this family of parallel hypersurfaces is a family of isoparametric hypersurfaces in the original sense, i.e., it is the family of level sets of an isoparametric function
(see, for example, \cite[pp. 86--91]{CR8} for more detail on this argument).

In summary, the analytic definition of an isoparametric family of hypersurfaces in terms of level sets of an isoparametric function on a real space form $\widetilde{M}^{n+1}$ 
is locally equivalent to the geometric definition of an isoparametric family as a
family of parallel hypersurfaces to a hypersurface 
with constant principal curvatures.

Therefore, in the case where the ambient space is a real space form $\widetilde{M}^{n+1}$,
we will say that  a connected hypersurface $M^n$ immersed in $\widetilde{M}^{n+1}$ is an {\em isoparametric hypersurface} if it has constant principal curvatures. \\

\begin{remark}{Isoparametric hypersurfaces of complex space forms}
\label{rem-isop-complex-space-forms}

\noindent
{\rm The definition of an isoparametric hypersurface in the paragraph above is not the appropriate formulation if 
$\widetilde{M}$  is only assumed to be a Riemannian manifold.  Rather the appropriate generalization is the following: a hypersurface $M$ in a Riemannian manifold $\widetilde{M}$ is {\em isoparametric} if and only if all nearby parallel hypersurfaces $M_t$ have constant mean curvature.  (See, for example, \cite[pp. 526--531]{CR8}.)
For example, when $\widetilde{M}$ is complex $n$-dimensional projective space ${\bf CP}^n$ or complex
$n$-dimensional hyperbolic space ${\bf CH}^n$, a
(real) isoparametric hypersurface may have nonconstant principal curvatures.  Nevertheless, making use of the Hopf map $\pi: S^{2n+1} \rightarrow {\bf CP}^n$ (see \cite[p. 346]{CR8}), one can show that an isoparametric hypersurface in $ {\bf CP}^n$ lifts to an isoparametric hypersurface in $S^{2n+1}$.  This facilitates the classification of isoparametric hypersurfaces in  ${\bf CP}^n$.
In fact, similar methods can be used to classify isoparametric hypersurfaces in ${\bf CH}^n$, although in that case one
must deal with Lorentzian isoparametric hypersurfaces in the anti-de Sitter space.
For further details about isoparametric hypersurfaces in complex space forms and more
general Riemannian manifolds,, see the papers of
Q.-M. Wang \cite{Wang-1a}--\cite{Wang2}, G. Thorbergsson \cite{Th7}, L. Xiao \cite{Xiao}, M. Dominguez-Vazquez et al. \cite{DV2}--\cite{DV4}, J. Ge, Z.-Z. Tang, and W. Yan \cite{Ge-Tang-Yan},
and the book of Cecil and Ryan \cite[pp. 343--531]{CR8}.}
\end{remark}

In Section 3 of his first paper on isoparametric hypersurfaces, Cartan \cite{Car1} derived what he called a
``fundamental formula,'' which is the basis for many of his results on the subject.
The formula involves  the distinct
principal curvatures $\lambda_1,\ldots,\lambda_g$, and their respective multiplicities $m_1,\ldots,m_g$, of an isoparametric
hypersurface $f:M^n \rightarrow \widetilde{M}^{n+1}(c)$ in a space form of constant sectional
curvature $c$.
If the number $g$ of distinct principal curvatures is greater than one, Cartan showed that for each $i$,
$1 \leq i \leq g$, the following equation holds,
\begin{equation}
\label{eq:1-car-id}
\sum_{j\neq i} m_j \frac{c + \lambda_i \lambda_j}{\lambda_i - \lambda_j} = 0.
\end{equation}
Cartan's original formulation of the formula was given in a more complicated form, but he gave the formulation in
equation (\ref{eq:1-car-id}) in his third paper \cite{Car3}.  Equation (\ref{eq:1-car-id}) is now known as 
``Cartan's formula,''  or ``Cartan's identity."  Another proof of this formula is given in \cite[pp. 91--96]{CR8}.

Using his formula, Cartan was able to classify isoparametric hypersurfaces in ${\bf R}^{n+1}$ and $H^{n+1}$.  In both cases, one can easily prove using Cartan's formula that the number $g$ of distinct principal curvatures must be either 1 or 2.   Then an elementary argument shows that a connected isoparametric hypersurface is either totally umbilic, i.e., $g=1$, or in the case $g=2$,
it is an open subset of a tube of constant radius over a totally geodesic submanifold of codimension greater than one (see, for example, \cite[pp. 96--98]{CR8}).

In the sphere $S^{n+1}$, however, Cartan's formula does not lead to the conclusion that $g \leq 2$, and in fact Cartan
himself produced examples with $g = 1, 2, 3$ and 4 distinct
principal curvatures (see \cite{Car4} for Cartan's examples with $g=4$).  
Moreover, he classified isoparametric hypersurfaces $M^n \subset S^{n+1}$ with
$g \leq 3$, as we will discuss below.  

Despite the depth and beauty of Cartan's work on this topic, 
however, it was largely
ignored until the subject was revived in the 1970's by 
Nomizu \cite{Nom1}--\cite{Nom2} and M\"{u}nzner  \cite{Mu}--\cite{Mu2}.
Among other things, Cartan showed that isoparametric hypersurfaces
come as a parallel family of hypersurfaces, i.e., if
$f:M \rightarrow S^n$ is an isoparametric hypersurface, then so is 
any parallel hypersurface $f_t$.  Of course, $f_t$ is not an
immersion if $\mu = \cot t$ is a principal curvature of $M$.
Then, however, $f_t$ factors through an immersion of the
space of leaves $M/T_{\mu}$  of
the principal foliation $T_{\mu}$.  Thus,
$f_t$ is a submanifold of codimension $m+1$ in $S^n$, where
$m$ is the multiplicity of $\mu$, and $f_t$ is called a {\em focal submanifold} of the isoparametric family
of hypersurfaces.

M\"{u}nzner \cite{Mu}--\cite{Mu2} showed that the principal curvatures of an
isoparametric hypersurface $M \subset S^{n+1}$ must have a very specific form, as follows
(see also \cite[p. 108]{CR8}).

\begin{theorem}
\label{thm:1-prin-curv-isop-hyp}
Let $M \subset S^{n+1}$ be a connected isoparametric hypersurface with $g$ principal curvatures
$\lambda_i = \cot \theta_i$, $0 < \theta_1 < \cdots < \theta_g < \pi$, with respective multiplicities $m_i$. Then
\begin{equation}
\label{eq:1-prin-curv-formula}
\theta_i = \theta_1 + (i-1) \frac{\pi}{g} , \quad 1 \leq i \leq g,
\end{equation}
and the multiplicities satisfy $m_i = m_{i+2}$ (subscripts mod $g$).  For any point $x \in M$, there
are $2g$ focal points of $(M,x)$ along the normal geodesic to $M$ through $x$, and they are evenly
distributed at intervals of length $\pi/g$.
\end{theorem}

M\"{u}nzner also showed that a
parallel family of isoparametric hypersurfaces in $S^n$ always
consists of the level sets in $S^n$ of a homogeneous polynomial $F$
defined on ${\bf R}^{n+1}$ satisfying certain differential equations on the
differential parameters $\Delta_1 F$ and $\Delta_2 F$. 
(This had been shown earlier by Cartan \cite{Car2} in the case where all of the
principal curvatures are assumed to have the same multiplicity.)

This implies that any local piece
of an isoparametric hypersurface can be extended to a unique
compact, connected isoparametric hypersurface.  M\"{u}nzner  showed further that
regardless of the number of distinct principal curvatures
of $M$, there are only two distinct focal submanifolds in a
parallel family of isoparametric hypersurfaces, and these are
minimal submanifolds of the sphere.  (The minimality of the
focal submanifolds was also established independently by Nomizu
\cite{Nom1}--\cite{Nom2}.) 

M\"{u}nzner  
used this information
and a difficult topological argument to prove the
following important result.  The key fact 
here is that any isoparametric hypersurface divides the sphere
into two ball bundles over the two focal submanifolds (see also \cite[pp. 131--137]{CR8}).

\begin{theorem}
\label{thm:munzner-g}
(M\"{u}nzner) The number $g$ of distinct principal curvatures of an isoparametric
hypersurface in $S^n$ must be $1,2,3,4$ or $6$.
\end{theorem}

Cartan classified isoparametric hypersurfaces with 
$g \leq 3$ principal curvatures, as follows.  Of course, if $g=1$,
then $M$ is totally umbilic, and it must be a great or small
sphere.  If $g=2$, then $M$ must be a standard product of two
spheres
\begin{displaymath}
S^k(r) \times S^{n-k-1}(s) \subset S^n, \quad r^2+s^2=1.
\end{displaymath}

In the case $g=3$, Cartan \cite{Car2}
showed that all the principal curvatures must have the same multiplicity
$m=1,2,4$ or 8, and the isoparametric hypersurface must be an open subset of a tube of
constant radius 
over a standard embedding of a projective
plane ${\bf FP}^2$ into $S^{3m+1}$,
where ${\bf F}$ is the division algebra
${\bf R}$, ${\bf C}$, ${\bf H}$ (quaternions),
${\bf O}$ (Cayley numbers), for $m=1,2,4,8,$ respectively.  Thus, up to
congruence, there is only one such family for each value of $m$.
This was a remarkable result with a difficult proof. These isoparametric
hypersurfaces with three principal curvatures are often referred to as
{\em Cartan hypersurfaces}.  (For more detail, see \cite[pp. 151--155]{CR8} or \cite{CR-cartans-examples}.)

All of Cartan's examples are homogeneous, and he asked in {\cite{Car3} if every isoparametric hypersurface
is homogeneous.  However,
that is not true as we will discuss below.

\begin{remark}{Homogeneous isoparametric hypersurfaces in spheres}
\label{rem-Takagi-Takahashi}

\noindent
{\rm In 1972, Takagi and Takahashi \cite{TT} gave a complete classification of all homogeneous
isoparametric hypersurfaces
in $S^{n+1}$, based on the work of Hsiang and Lawson \cite{HL}.  Takagi and Takahashi showed that each homogeneous
isoparametric hypersurface in $S^{n+1}$ is a principal orbit of the isotropy representation
of a Riemannian symmetric space of 
rank 2, and they gave a complete list of examples \cite[p. 480]{TT}.  This list contains examples with $g=6$
principal curvatures as well as those with $g=1,2,3,4$ principal curvatures. }
\end{remark} 

As noted above, Cartan classified isoparametric hypersurfaces in spheres with $g \leq 3$ principal curvatures.
In the case of $g=4$ principal curvatures, the classification 
is now complete due to the work of several 
mathematicians, as we now describe.  First,
M\"{u}nzner  \cite{Mu}--\cite{Mu2}  proved that
the principal curvatures can have at most two
distinct multiplicities $m_1,m_2$. 
Then Ferus, Karcher and M\"{u}nzner
\cite{FKM} 
used representations
of Clifford algebras to construct
for any positive integer
$m_1$ an infinite series of isoparametric hypersurfaces
with four principal curvatures having respective multiplicities
$(m_1,m_2)$, where $m_2$ is nondecreasing and
unbounded in each series. 
These examples are now known as isoparametric hypersurfaces of {\it OT-FKM-type}, because
the construction of Ferus, Karcher and M\"{u}nzner was a generalization of an earlier construction due to Ozeki and Takeuchi \cite{OT}--\cite{OT2}.  Many of the hypersurfaces of  OT-FKM-type are inhomogeneous, 
which was first shown by Ozeki and Takeuchi.

Stolz \cite{Stolz}  next proved that the multiplicities $(m_1,m_2)$
of the principal curvatures of an isoparametric
hypersurface with $g=4$ principal curvatures
must be the same as those of an isoparametric hypersurface of
OT-FKM-type, or else $(2,2)$ or $(4,5)$, which are the multiplicities for certain homogeneous examples
that are not of OT-FKM-type  (see the classification 
of homogeneous isoparametric hypersurfaces of Takagi and Takahashi \cite {TT} mentioned in 
Remark \ref{rem-Takagi-Takahashi}).  The results of Stolz generalized the previous work on the multiplicities by
several mathematicians, including M\"{u}nzner \cite{Mu}--\cite{Mu2},
Abresch \cite{Ab}, Grove and Halperin \cite{GH}, Tang \cite{Tang}--\cite{Tang2} and Fang
\cite{Fang3}--\cite{Fang2}.

Cecil, Chi and Jensen \cite{CCJ1} then
showed that if the multiplicities of an isoparametric hypersurface with four principal curvatures satisfy the condition
$m_2 \geq 2 m_1 - 1$, then
the hypersurface is of OT-FKM-type.  (A different proof of this result, using isoparametric 
triple systems, was given later by Immervoll \cite{Im}.)

Taken together with known results of Takagi \cite{Takagi} for $m_1 = 1$,
and Ozeki and Takeuchi \cite{OT}--\cite{OT2} for $m_1 = 2$, this result of Cecil, Chi and Jensen handled all
possible pairs of multiplicities except for four cases, the OT-FKM pairs
$(3,4), (6,9)$ and $(7,8)$,and  the homogeneous pair $(4,5)$.

In a subsequent series of papers, Chi \cite{Chi}--\cite{Chi4} completed the classification of isoparametric hypersurfaces with four principal curvatures.
Specifically, Chi showed that in the cases of multiplicities
$(3,4)$, $(6,9)$ and $(7,8)$, the isoparametric hypersurface must be of OT-FKM-type, and in the case $(4,5)$, it must be homogeneous.

In summary, the final conclusion is that an isoparametric hypersurface with $g=4$ principal curvatures must either be of
OT-FKM-type, or else a homogeneous isoparametric hypersurface with multiplicities $(2,2)$ or $(4,5)$.

In the case of an isoparametric hypersurface with $g=6$ principal curvatures, M\"{u}nzner \cite{Mu}--\cite{Mu2} showed
that all of the principal curvatures must have the same multiplicity $m$, and
Abresch \cite{Ab} showed that $m$ must equal 1 or 2.
By the classification of homogeneous isoparametric hypersurfaces due to Takagi and Takahashi \cite{TT},
there is up to congruence only one homogeneous family in each case, $m=1$ or $m=2$.  

These homogeneous examples have been shown to be
the only isoparametric hypersurfaces in the case $g=6$ by Dorfmeister and Neher \cite{DN5}
in the case of multiplicity $m=1$, and by Miyaoka \cite{Mi11}--\cite{Mi12} in the case $m=2$
(see also the papers of Siffert \cite{Siffert1}--\cite{Siffert2}, and Miyaoka \cite{Mi25}).
This completes the classification of isoparametric hypersurfaces in spheres.

In a notable series of papers, Dorfmeister and Neher \cite{DN1}--\cite{DN6}
gave an algebraic approach to the study of isoparametric
hypersurfaces and isoparametric triple systems. 
For general surveys on isoparametric hypersurfaces in spheres,
see the papers of Thorbergsson \cite{Th6}--\cite{Th7}, Cecil \cite{Cec10}, and Chi \cite{Chi-survey}.
For a survey of various contexts in Riemannian geometry in which isoparametric hypersurfaces have appeared,
see Cecil-Ryan \cite[pp. 182--184]{CR8}.

\begin{remark}{Isoparametric hypersurfaces are taut}
\label{rem:isoparametic-hyp-taut}

\noindent
{\rm All isoparametric hypersurfaces in $S^n$ and their focal submanifolds are taut.  This was
established by Cecil and Ryan \cite{CRBer}
using the results of  M\"{u}nzner \cite{Mu}--\cite{Mu2}.
In particular, any normal geodesic to an isoparametric hypersurface
is also normal to each parallel hypersurface and to the focal
submanifolds.  Using  M\"{u}nzner's results, one can show that 
every non-focal point $p \in S^n$ lies on exactly $2g$ normal
geodesics from points in $M$, i.e., the height function
$l_p$ has exactly $2g$ critical points on $M$.  Since
M\"{u}nzner showed that the sum of the ${\bf Z}_2$-Betti numbers
of an isoparametric hypersurface with $g$ principal curvatures
is $2g$, the hypersurface $M$ is taut.  A similar argument shows
that the focal submanifolds are taut.}
\end{remark}

\begin{remark}{Isoparametric submanifolds of higher codimension}
\label{rem:isoparametric-higher-codimension}

\noindent
{\rm In the early 1980's, a theory of isoparametric submanifolds of
codimension greater than one was introduced independently
by several mathematicians, including
Carter and West \cite{CW4}, \cite{CW6}--\cite{CW8}, West \cite{West}--\cite{West1}, 
Terng \cite{Te1}--\cite{Te5},  Hsiang, Palais and Terng \cite{HPT}, Str\"{u}bing \cite{Str}, and Harle \cite{Har}.

By definition,
a connected, complete submanifold $V$ in a real space form $\widetilde{M}^n$ is said to be {\em isoparametric} if it has 
flat normal bundle
and if for any parallel section of the unit normal bundle $\eta:V \rightarrow B^{n-1}$, the principal curvatures of
the shape operator $A_\eta$ are constant.

In the decade after this
definition was formulated, intense
research by several mathematicians produced a remarkable
theory (see Terng \cite{Te5} and Palais-Terng
\cite{PalT} for more detail). 
Among other things,
Terng \cite{Te1} showed that a compact isoparametric submanifold
in Euclidean space must be taut and lie in a standard hypersphere in the Euclidean space.
Palais and Terng \cite{PalT1} showed that the only homogeneous
isoparametric submanifolds in Euclidean spaces are the
principal orbits of the isotropy representations of
symmetric spaces, which had been studied extensively by Bott 
and Samelson \cite{BS}, who showed that the orbits are taut.
The orbits of the isotropy 
representations of symmetric spaces (also known as standard embeddings of $R$-spaces or generalized flag manifolds),
were studied independently
by Takeuchi and Kobayashi \cite{TK}  
who also showed that
they were taut.  (See Thorbergsson \cite{Th8} for a survey on the classical symmetric $R$-spaces from
the point of view of projective and polar geometry.)

The class of $R$-spaces
contains many special subclasses of homogeneous submanifolds which
were shown to be taut by various special arguments.  (See, for
example, the papers of Kobayashi \cite{Kob}, Tai \cite{Tai}, Wilson
\cite{Wil}, Kuiper \cite{Ku3}, Ferus \cite{Fer} and K\"{u}hnel
\cite{Kuh}.) 
Later, Hsiang, Palais and Terng \cite{HPT} obtained many
facts about the geometry and topology of isoparametric submanifolds,
including the fact that they and their focal submanifolds are taut.

This line of research culminated with the work of
Thorbergsson \cite{Th2}, who used the theory of Tits buildings to show that all irreducible 
isoparametric submanifolds which are substantially embedded in $S^n$ with codimension greater than one are homogeneous.
Thus, they are
principal orbits of isotropy representations
of symmetric spaces, as was the case for homogeneous isoparametric hypersurfaces in the sphere.
Subsequent to Thorbergsson's paper, Olmos \cite{Olm}, and Heintze and Liu \cite{Heintze-Liu} published 
alternate proofs of Thorbergsson's result.  

In a further generalization,
Heintze, Olmos and Thorbergsson \cite{HOT} defined a submanifold $\phi:V \rightarrow {\bf R}^n$  (or $S^n$) 
to have {\em constant principal curvatures} if for any smooth
curve $\gamma$ on $V$ and any parallel normal vector field
$\xi(t)$ along $\gamma$, the shape operator $A_{\xi(t)}$ has
constant eigenvalues along $\gamma$.  If the normal bundle
$N(V)$ is flat, then having constant principal curvatures is
equivalent to being isoparametric.  They then
showed that a submanifold with constant principal curvatures is
either isoparametric or a focal submanifold of an isoparametric
submanifold.
The survey paper of Thorbergsson \cite{Th6} has a good account of all the topics mentioned in this remark.}
\end{remark}

\begin{remark}{Isoparametric hypersurfaces in semi-Riemannian spaces}
\label{rem:isop-hyp-pseudo-Riemann}

\noindent
{\rm Nomizu \cite{Nom3} initiated the study of isoparametric hypersurfaces in semi-Riemannian space forms
by proving a generalization of Cartan's formula for spacelike hypersurfaces in a Lorentzian space form
$\widetilde{M}^n_1(c)$ of constant sectional curvature $c$.  As a consequence of this formula, Nomizu showed that
a spacelike isoparametric hypersurface in $\widetilde{M}^n_1(c)$ can have at most two distinct principal curvatures
if $c \geq 0$. Later Li and Xie \cite{LX} proved that this conclusion also holds for spacelike isoparametric
hypersurfaces in $\widetilde{M}^n_1(c)$ for $c<0$.
Magid \cite{Mag} studied isoparametric hypersurfaces in Lorentz space whose shape operator is not diagonalizable,
and Hahn \cite{Hahn} did an extensive study of isoparametric hypersurfaces in semi-Riemannian space forms of arbitrary signatures.}
\end{remark}

\section{Dupin Submanifolds}
\label{sec:dupin-submanifolds}

In this section, we introduce the notion of Dupin submanifolds,
beginning with hypersurfaces.  Let $f:M \rightarrow {\bf R}^n$
be an immersed hypersurface.  Let $\xi$ be a locally defined
field of unit normals to $f(M)$.
A {\em curvature surface} of $M$ is a smooth submanifold $S$
such that for each point $x \in S$, the tangent space
$T_xS$ is equal to a principal space of the shape operator
$A$ of $M$ at $x$.  This generalizes the classical notion of a line of curvature on a surface in ${\bf R}^3$.  

Pinkall \cite{P2} formulated the notion of a Dupin hypersurface as follows.

\begin{definition}
\label{def:dupin}
{\rm A hypersurface $M$ in $ {\bf R}^n$ is said to be {\it Dupin} if:\\

\noindent
(a) along each curvature surface, the corresponding principal
curvature is constant.}\\

\noindent
{\rm A Dupin hypersurface $M$ is called {\it proper Dupin} if, in addition 
to Condition (a), the following condition is satisfied:\\

\noindent
(b) the number $g$ of distinct principal curvatures is constant on $M$.}
\end{definition}

Several remarks about these definitions are in order.  The proofs can
be found in \cite[pp. 132--151]{CR7} or \cite[pp. 24--32]{CR8}.
First, if the dimension of a curvature surface $S$
is greater than one, then the corresponding principal curvature
is automatically constant on $S$.  This is proven using the
Codazzi equation.  Second, 
Condition (b) is equivalent to requiring that each continuous
principal curvature function has constant multiplicity on $M$.

For any hypersurface $M$ in ${\bf R}^n$, there is an open dense subset of $M$ on
which the multiplicities of the continuous principal curvatures
of $M$ are locally constant. (See, for example, Singley \cite{Sin}.)
Suppose now that a continuous principal curvature $\mu$ has
constant multiplicity $m$ on an open subset $U \subset M$.
Then $\mu$ and its principal distribution $T_{\mu}$ are smooth on
$U$.  Furthermore, again using the Codazzi equation, one can show
that $T_{\mu}$ is integrable, and thus it is called the
{\em principal foliation} corresponding to $\mu$.  The leaves
of this principal foliation are the curvature surfaces 
corresponding to $\mu$ on $U$.  The principal curvature
$\mu$ is constant along each of its curvature surfaces in $U$
if and only if these curvature surfaces are open subsets
of $m$-dimensional Euclidean spheres or planes.  

The {\em focal map} $f_{\mu}$ corresponding to $\mu$ is the map
which maps a point $x\in M$ at which $\mu \neq 0$ to the focal point $f_{\mu}(x)$ corresponding to
$\mu$, i.e.,
\begin{equation}
\label{eq:focal-map-Euclidean}
f_{\mu}(x) = f(x) + \frac{1}{\mu(x)} \xi(x).
\end{equation}
The domain of $f_\mu$ is the subset of $M$ on which $\mu \neq 0$.

A direct calculation shows that $\mu$ is constant
along each of its curvature surfaces in the set $U$ defined above  if and only if
the focal map $f_{\mu}$ factors through an immersion of the
$(n-1-m)$-dimensional space of leaves $U/T_{\mu}$ into ${\bf R}^n$.

In summary, on an open subset $U$ on which the number of distinct
principal curvatures is constant, Condition (a) is equivalent
to requiring that each curvature surface in each principal
foliation be an open subset of a Euclidean sphere or plane
with dimension equal to the multiplicity of the corresponding
principal curvature. On the set $U$, Condition (a) is also equivalent
to the condition that each focal map is a submanifold of
codimension greater than one in ${\bf R}^n$.

Like tautness, both the Dupin and proper Dupin conditions are
invariant under M\"{o}bius transformations and under stereographic
projection from $S^n$ to ${\bf R}^n$ or from $H^n$ to ${\bf R}^n$
(see \cite[pp. 17--18]{Cec9}).  These Dupin conditions are
also invariant under parallel transformations,  and,  more generally, under
the group of Lie sphere transformations.  
A proof of these
claims is best formulated in the setting of Lie sphere geometry
(see Section \ref{sec:submanifolds-lie-geometry},  or \cite{P2}, \cite{Cec9},  for more detail.)

An important class of compact proper Dupin hypersurfaces in
${\bf R}^n$ is obtained by taking the images under stereographic
projection of isoparametric hypersurfaces in $S^n$.  Of course,
for these examples, the number $g$ of distinct principal curvatures
must be $1,2,3,4$ or 6.  In fact, Thorbergsson \cite{Th1} 
showed that this restriction on $g$ holds for any compact
proper Dupin hypersurface embedded in ${\bf R}^n$, as follows.

First, Thorbergsson \cite{Th1} showed that if $M$ is a complete, connected proper
Dupin hypersurface embedded in a real space form $\widetilde{M}^n$, then $M$ is tautly
embedded by proving the following theorem.

\begin{theorem}
\label{thm:proper-Dupin-implies-taut} 
Let $M \subset \widetilde{M}^n$ be a complete, connected proper Dupin hypersurface embedded in a real space form 
$\widetilde{M}^n$.  Then $M$ is taut with respect to ${\bf Z}_2$-coefficients.
\end{theorem}

We now briefly discuss Thorbergsson's method of proof.
Let $p \in \widetilde{M}$ and let $L_p$ be the distance function $L_p (x) = d(p,x)^2$, where $d(p,x)$ is the distance from $p$ to $x$ in $\widetilde{M}$.  By Sard's Theorem,
the restriction of $L_p$ to $M$ is a Morse function for almost all $p \in \widetilde{M}$.  
To prove that $M$ is tautly
embedded, one must show that every non-degenerate (Morse) function of the form $L_p$ has the minimum number of critical points required by the Morse inequalities on $M$.

Thorbergsson's method of proving tautness in this theorem is different than the
method that was used to prove that isoparametric hypersurfaces in spheres are taut,
as described in Remark \ref{rem:isoparametic-hyp-taut}.
In proving Theorem \ref{thm:proper-Dupin-implies-taut} above, Thorbergsson does not
assume knowledge of the homology of the proper Dupin hypersurface $M$. 
Rather, he uses the principal foliations to
construct concrete ${\bf Z}_2$-cycles in $M$, which enable him to show
that every critical point of every non-degenerate distance
function is of linking type (see, Morse-Cairns \cite[p. 258]{MC}), and thus $M$ is taut. 
See Thorbergsson's paper \cite{Th1} for a complete proof.

\begin{remark}{Focal submanifolds of Dupin hypersurfaces may not be taut}
\label{focal-submanifolds-not-taut}

\noindent
{\rm In contrast to the case of an isoparametric hypersurface (see Remark \ref{rem:isoparametic-hyp-taut}), the focal
submanifolds of a compact, connected proper Dupin hypersurface embedded in $S^n$ or ${\bf R}^n$ are
not necessarily taut.
For example, if $M \subset {\bf R}^3$ is 
a compact ring cyclide of Dupin \cite[pp. 148--159]{Cec9}  that is not a torus of revolution,  then one focal
submanifold of $M$ is an ellipse, which is not taut in ${\bf R}^3$, although it is tight in ${\bf R}^3$  (see also
Buyske  \cite{Bu}).}
\end{remark} 

In order to be able to apply the results of M\"{u}nzner \cite{Mu}--\cite{Mu2},
we now focus on the case where 
$M$ is a compact, connected proper Dupin hypersurface in the sphere $S^n$.
Thorbergsson \cite{Th1} used Theorem \ref{thm:proper-Dupin-implies-taut} to derive the following
important theorem. Here $M_+$ and $M_-$ are the first focal
submanifolds on either side of $M$ in $S^n$.  (See also \cite[pp. 140--143]{CR8} for a detailed proof.)

\begin{theorem}
\label{thm:dupin-ball-bundle} 
Let $M \subset S^n$ be a compact, connected proper Dupin hypersurface. Then
$M$ divides $S^n$ into to a union of two ball bundles over the two focal
submanifolds $M_+$ and $M_-$.  
\end{theorem}

Thorbergsson could then invoke M\"{u}nzner's 
\cite{Mu}--\cite{Mu2} results to obtain the following theorem. 

\begin{theorem}
\label{thm:dupin-restriction-g}
Let $M$ be a compact proper Dupin hypersurface in $S^n$.
Then\\
$(1)$ the number $g$ of distinct principal curvatures of $M$ is
$1,2,3,4$ or $6$.\\
$(2)$ The sum of the ${\bf Z}_2$-Betti numbers of $M$ is $2g$.
\end{theorem}

Thorbergsson \cite[p. 980]{Th6} noted that his construction of the cycles in \cite{Th1}
can be used to show that the multiplicities of the principal curvatures of a compact, 
connected proper Dupin hypersurface
satisfy the relation
 \begin{equation}
 \label{eq:subscripts-dupin}
 m_k = m_{k+2} \quad {\rm (subscripts\  mod}\  g),
 \end{equation}
 as in M\"{u}nzner's Theorem \ref{thm:1-prin-curv-isop-hyp} above for isoparametric hypersurfaces.  The proof here
 is different than the proof in the isoparametric case.  In the Dupin case,
it is accomplished by calculating the homology of $M$ by using the cycles constructed by Thorbergsson 
in proving Theorem \ref{thm:proper-Dupin-implies-taut} and comparing
that with the calculation of the homology based on the fact that $M$ divides $S^n$ into two ball bundles,
as in Theorem \ref{thm:dupin-ball-bundle}. 

Furthermore, the following theorem regarding the multiplicities of a compact, connected proper Dupin hypersurface embedded in $S^{n+1}$ has also been proven.
\begin{theorem}
\label{thm:mult-compact-Dupin}
For each value of $g = 1,2,3,4,6$, the possible multiplicities of 
the principal curvatures of a
compact, connected proper Dupin hypersurface embedded in $S^n$ are the same as the possible multiplicities
for an isoparametric hypersurface in $S^n$ with the same number of principal curvatures.
\end{theorem}
This was shown in the case $g=4$ by Stolz \cite{Stolz} and for
$g=6$ by Grove and Halperin \cite{GH}. Both of these papers involve sophisticated topological arguments.

Note that for
$g=1$ and $g=2$, there are no restrictions on the multiplicities. In the case $g=3$,
Miyaoka \cite{Mi1} proved that a compact, connected proper Dupin hypersurface $M$ is 
Lie equivalent to an isoparametric hypersurface, and thus it has the same multiplicities as an isoparametric hypersurface.
That is, all the multiplicities are equal to a certain integer $m$, and $m$ is $1,2,4$ or 8.

\subsubsection*{Pinkall's local constructions of proper Dupin hypersurfaces}

In contrast with Thorbergsson's Theorem \ref{thm:dupin-restriction-g},
Pinkall \cite{P2} (see also \cite[pp. 34--35]{CR8} or \cite[pp. 125--141]{Cec9}) 
showed how to locally construct a proper Dupin hypersurface in Euclidean space with an
arbitrary number of distinct principal curvatures having any
prescribed multiplicities. This is done using the following basic constructions of Pinkall.

\begin{definition}[Pinkall's local constructions]
\label{def:pinkalls-constructions}
Start with a Dupin hypersurface
$W^{n-1}$ in ${\bf R}^n$ and then consider ${\bf R}^n$
as the linear subspace ${\bf R}^n \times \{ 0 \}$
in ${\bf R}^{n+1}$. 
The following 
constructions yield a Dupin hypersurface $M^n$ in
${\bf R}^{n+1}$.
\begin{enumerate}
\item[$(1)$] Let $M^n$ be the cylinder $W^{n-1} \times {\bf R}$ in
${\bf R}^{n+1}$.
\item[$(2)$] Let $M^n$ be the hypersurface in ${\bf R}^{n+1}$
obtained by rotating
$W^{n-1}$ around an axis
${\bf R}^{n-1} \subset {\bf R}^n$.
\item[$(3)$] Project $W^{n-1}$ stereographically onto a hypersurface
$V^{n-1} \subset S^n \subset {\bf R}^{n+1}$.  Let
$M^n$ be the cone over $V^{n-1}$ in ${\bf R}^{n+1}$.
\item[$(4)$] Let $M^n$ be a tube of constant radius in ${\bf R}^{n+1}$ around $W^{n-1}$.
\end{enumerate}
\end{definition}

In general, these constructions introduce a new principal curvature 
of multiplicity one on $M^n$ which is constant along its lines 
of curvature.  The other principal curvatures are determined by the 
principal curvatures of $W^{n-1}$, and the Dupin property is preserved
for these principal curvatures.  Thus, 
if $W^{n-1}$ is a proper Dupin hypersurface in ${\bf R}^n$
with $g$ distinct principal curvatures, then in general, $M^n$ is a 
proper Dupin hypersurface in ${\bf R}^{n+1}$ with $g+1$ distinct principal curvatures.

However, there are cases where the number
of distinct principal curvatures of $M^n$ does not equal $g+1$ at some points, as we will discuss after the proof of
Theorem \ref{thm:4.1.1} below.   These constructions can be
generalized to produce a new principal curvature of multiplicity
$m$ by considering ${\bf R}^n$ as a subset of 
${\bf R}^n \times {\bf R}^m$ rather than ${\bf R}^n \times {\bf R}$.

Although Pinkall gave these four constructions, he showed \cite[p. 438]{P2} 
that the cone construction $(3)$ is redundant in Lie sphere geometry, since it
is Lie equivalent
to the tube construction (see also \cite[p. 144]{Cec9}). Thus, we often work with just the cylinder, 
surface of revolution, and tube constructions, as in \cite{Cec9}.

Using these constructions, Pinkall \cite{P2} showed how
to locally produce a proper Dupin hypersurface in Euclidean space with an
arbitrary number of distinct principal curvatures, each with any given multiplicity as follows.

\begin{theorem}
\label{thm:4.1.1}
Given positive integers $m_1, \ldots , m_g$, with 
$m_1 + \cdots + m_g = n-1$, there exists a proper
Dupin hypersurface $M^{n-1}$ in ${\bf R}^n$ with $g$
distinct principal curvatures having respective
multiplicities $m_1, \ldots , m_g$.
\end{theorem} 
\begin{proof}
The proof is by an inductive construction
which will be clear once the first few cases are handled.
First we construct a proper Dupin hypersurface $M^3$ in 
${\bf R}^4$ with three principal curvatures of multiplicity
one. Begin with an open subset $U$ of a torus of
revolution in ${\bf R}^3$ on which neither principal
curvature vanishes.  Take $M^3$ to be the cylinder
$U \times {\bf R}$ in ${\bf R}^3 \times {\bf R} =
{\bf R}^4$.  Then $M^3$ has three distinct principal
curvatures at each point, one of which is identically
zero.  These are clearly constant along their
corresponding lines of curvature.
Next, to construct a proper Dupin hypersurface in ${\bf R}^5$
with three distinct principal curvatures having
respective multiplicities 1, 1, 2, take a cylinder
$U \times {\bf R}^2$ in ${\bf R}^3 \times {\bf R}^2$. Finally,
to get a proper Dupin hypersurface $V^4$ in
${\bf R}^5$ with four principal curvatures of multiplicity 1, first 
invert the hypersurface $M^3$ constructed above in a 3-sphere
in ${\bf R}^4$ chosen so that the image of $M^3$
contains an open set $W^3$ on which no principal
curvature vanishes.  The hypersurface $W^3$ is proper Dupin, since the proper Dupin property is preserved by
M\"{o}bius transformations.
Now take $V^4$ to be the
cylinder $W^3 \times {\bf R}$. 
\end{proof}

\begin{remark}{Pinkall's Constructions and Compactness}
\label{rem:pinkall-constructions-compactness}

\noindent
{\rm While Pinkall's constructions locally produce proper Dupin hypersurfaces, 
there are often problems in trying to produce compact proper Dupin hypersurfaces by using these
constructions.  We now discuss some of the problems involved
with the cylinder, 
surface of revolution, and tube constructions individually (see \cite[pp. 127--141]{Cec9} for more details).

For the cylinder construction $(1)$,
the new principal curvature of $M^n$ is identically zero, while 
the other principal curvatures of $M^n$ are equal to those
of $W^{n-1}$.  Thus, if one of the principal curvatures $\mu$
of $W^{n-1}$ is zero at some points but not identically
zero, then the number of distinct principal curvatures 
is not constant on $M^n$, and so $M^n$ is Dupin but not proper Dupin.  

For the surface of revolution construction $(2)$, if the focal point corresponding to a principal curvature
$\mu$ at a point $x$ of the profile submanifold $W^{n-1}$ lies on the axis of revolution ${\bf R}^{n-1}$, then
the principal curvature of $M^n$ at $x$ 
determined by $\mu$ is equal to the new principal curvature of $M^n$ resulting from
the surface of revolution construction.
Thus, if the focal point of $W^{n-1}$ corresponding to $\mu$ lies
on the axis of revolution for some, but not all points of  $W^{n-1}$, then $M^n$ is not proper Dupin.

If $M^n$ is a tube in ${\bf R}^{n+1}$ of radius
$\epsilon$ over $W^{n-1}$ obtained by the tube construction $(4)$, then there are exactly two
distinct principal curvatures at the points in the set
$W^{n-1} \times \{ \pm \epsilon \}$ in $M^n$, regardless
of the number of distinct principal curvatures on $W^{n-1}$.
Thus, $M^n$ is not a proper Dupin hypersurface, unless the original hypersurface $W^{n-1}$ is totally umbilic, i.e., it
has only one distinct principal curvature at each point.

Another problem with these constructions is that they may not yield
an immersed hypersurface in ${\bf R}^{n+1}$. In the tube construction, if the
radius of the tube is the reciprocal of one of the 
principal curvatures of $W^{n-1}$ at some point, then the constructed object has a singularity.  For the
surface of revolution construction, a singularity occurs
if the profile submanifold $W^{n-1}$ intersects the axis of revolution. 
Some of the issues mentioned in this Remark can be resolved by working
in the context of Lie sphere geometry (see \cite[pp. 125--148]{Cec9}).}
\end{remark}

\begin{remark}{Dupin submanifolds of higher codimension}
\label{rem:dupin-submanifolds-of-higher-codimension}

\noindent
{\rm As in the isoparametric case (see Remark \ref{rem:isoparametric-higher-codimension}),
the Dupin property can also be formulated for submanifolds of 
codimension greater than one.  To do this, we first need the definition
(due to Reckziegel \cite{Reck})
of a curvature surface in that case.  Suppose that
$\phi:V \rightarrow {\bf R}^n$ is a submanifold of
codimension greater than one, and let $B^{n-1}$ denote the
unit normal bundle of $\phi(V)$.  Then a {\em curvature
surface} is a connected submanifold $S \subset V$ for which
there is a parallel section $\eta : S \rightarrow B^{n-1}$
such that for each $x \in S$, the tangent space $T_xS$ is
equal to some smooth eigenspace of the shape operator
$A_{\eta}$. 
We then define $\phi(V)$ to be {\em Dupin}
if along each curvature surface, the corresponding principal
curvature of $A_{\eta}$ is constant.  The Dupin
submanifold $\phi(V)$ is {\em proper Dupin} if the number
of distinct principal curvatures of $A_{\xi}$ is constant as $\xi$ varies over
the unit normal bundle $B^{n-1}$.  From the definitions,
it is clear that an isoparametric submanifold is always
Dupin, but it may not be proper Dupin.  (See Terng
\cite[pp. 464--469]{Te5} for more discussion of this.)  Dupin submanifolds
of codimension greater than one
are handled quite naturally in the setting
of Lie sphere geometry (see Section \ref{sec:submanifolds-lie-geometry}).}
\end{remark}

Pinkall \cite[p. 439]{P2} showed that every extrinsically 
symmetric submanifold of a real space form is Dupin.  Takeuchi
\cite{Tak} then determined which of these are proper Dupin. \\

\begin{remark}{Proper Dupin hypersurfaces are algebraic}
\label{rem:dupin-algebraic}

\noindent
{\rm Another important result is that like isoparametric hypersurfaces,
proper Dupin hypersurfaces are algebraic in a certain sense.  
This was formulated by Cecil, Chi and Jensen \cite{CCJ3} as follows.}

\begin{theorem}
\label{thm:Dupin-algebraic}
Every connected proper Dupin hypersurface 
$f:M \rightarrow{\bf R}^n$ embedded in ${\bf R}^n$ is contained in a connected component of an 
irreducible algebraic subset of ${\bf R}^n$ of dimension $n-1$. 
\end{theorem}

{\rm The main idea of the proof of this theorem is due to
Pinkall, who sent a letter \cite{P6} to T. Cecil in 1984 that contained a sketch of the proof.  
However, Pinkall did not publish 
a proof, and a full proof based on Pinkall's sketch 
was not published until 2008 by Cecil, Chi and Jensen \cite{CCJ3}.
The proof makes use of the various principal foliations whose leaves are open subsets of spheres to construct an analytic algebraic parametrization of a 
neighborhood of $f(x)$ for each point $x \in M$. From this, one can get the final conclusion stated above
by using methods of real algebraic geometry.

In contrast to the situation for isoparametric hypersurfaces, however, a connected proper Dupin hypersurface 
in ${\bf R}^n$ or $S^n$ does not necessarily lie in a compact, connected proper Dupin hypersurface, 
as the following example illustrates.
Let $M^3$ be a tube of sufficiently small radius over a torus
$T^2 \subset {\bf R}^3 \subset {\bf R}^4$ so that $M^3$ is a compact, connected hypersurface in ${\bf R}^4$.
The tube $M^3$ is Dupin, not proper Dupin, 
since there are only two distinct principal
curvatures on the set $T^2 \times \{\pm \epsilon\}$ in $M^3$, but three distinct
principal curvatures on the complementary open subset $U \subset M^3$.   
Each connected open subset of  $U$ is a proper Dupin hypersurface in ${\bf R}^4$
with $g=3$ principal curvatures that is contained in the compact, connected algebraic hypersurface $M^3$. 
However, $M^3$ itself is only Dupin, and not proper Dupin (see \cite[p. 69]{Cec9} for more details).

The algebraicity,
and hence analyticity, of proper Dupin hypersurfaces was useful in clarifying certain points in the
2007 paper \cite{CCJ2} of Cecil, Chi and Jensen on proper Dupin hypersurfaces with four principal curvatures.

In a recent preprint, Wiesendorf \cite{Wiesendorf1} proves several important results concerning the analyticity and tautness of smooth Dupin hypersurfaces in Euclidean spaces.}
\end{remark}

\section{Relationship between Taut and Dupin}
\label{sec:relationship-taut-dupin}

In this section, we examine the relationship between the taut
and Dupin conditions for submanifolds of ${\bf R}^n$.  Many of these results also apply to submanifolds of the
other real space forms.
(See \cite[pp. 65--74]{CR8} for more details on these results.)

As noted in Section \ref{sec:dupin-submanifolds},
Thorbergsson \cite{Th1} showed in Theorem \ref{thm:proper-Dupin-implies-taut}
that if $M$ is a complete, connected proper
Dupin hypersurface embedded in a real space form $\widetilde{M}^n$, then $M$ is taut with respect to 
${\bf Z}_2$-coefficients.

Pinkall \cite{P4} then extended Thorbergsson's theorem
 to proper Dupin submanifolds of higher codimension in ${\bf R}^n$, 
as defined in Remark \ref{rem:dupin-submanifolds-of-higher-codimension}.
In particular, Pinkall  showed if $M$ is an embedded submanifold of ${\bf R}^n$
and $M_{\epsilon}$ is a tube of sufficiently small radius $\epsilon$
so that $M_{\epsilon}$ is
embedded in ${\bf R}^n$, then $M$ is taut with respect to 
${\bf Z}_2$-coefficients if and only if $M_{\epsilon}$ is taut
with respect to ${\bf Z}_2$-coefficients.  This can be combined
with Thorbergsson's result to yield the following.
\begin{theorem}
\label{thm:taut-codimension-greater-than-1}
Let $M$ be a compact, connected proper Dupin submanifold of codimension $m \geq 1$ embedded in ${\bf R}^n$.
Then $M$ is taut with respect to ${\bf Z}_2$-coefficients.
\end{theorem}

In the opposite direction of Theorem \ref{thm:taut-codimension-greater-than-1},
Pinkall \cite{P4}, and independently Miyaoka \cite{Mi8} (in the case of codimension one), proved the following
theorem, which is also valid for submanifolds of $S^n$. 
(See  \cite[pp. 68--71]{CR8} for a proof.)

\begin{theorem}
\label{thm:taut-implies-dupin} 
Every taut connected submanifold in ${\bf R}^n$ is Dupin (but not necessarily proper Dupin).
\end{theorem}

\begin{remark}{Taut does not imply proper Dupin}
\label{rem:dupin-not-proper}

\noindent
{\rm Note that a taut
submanifold need not be proper Dupin.  For example, as noted in Section \ref{sec:dupin-submanifolds}, 
a tube $M^3 \subset {\bf R}^4$ of sufficiently small radius $\epsilon$ over a torus of
revolution $T^2 \subset {\bf R}^3 \subset {\bf R}^4$ is taut,
but not proper Dupin, since there are only two distinct principal
curvatures on the set $T^2 \times \{\pm \epsilon\}$, but three distinct
principal curvatures elsewhere on $M$ (see Pinkall \cite{P2}, or Cecil \cite[p. 69]{Cec9}).  
Pinkall's other constructions can also be used to produce taut 
hypersurfaces that are Dupin, but not proper Dupin (see  Remark \ref{rem:pinkall-constructions-compactness}
 and \cite[pp. 185--190]{CR7}).}
\end{remark}

Of course, if $M \subset {\bf R}^n$  is a compact, connected hypersurface on which the number $g$ 
of distinct principal curvatures is assumed to be constant, then $M$ is Dupin implies that $M$ is proper Dupin. 
In that case, Theorems \ref{thm:taut-codimension-greater-than-1} and  \ref{thm:taut-implies-dupin}
imply that $M$ is taut if and only if $M$ is proper Dupin.  Thus, we have the following theorem, which was 
first proven for compact surfaces in ${\bf R}^3$ by Banchoff \cite{Ban1}.

\begin{theorem}
\label{thm:taut-dupin}
Let  $M \subset {\bf R}^n$ be a compact, connected hypersurface on which the number $g$ of distinct principal curvatures is constant.  Then $M$ is taut if and only if $M$ is proper Dupin.
\end{theorem}

If $M$ is a compact manifold, we can use a theorem
of Ozawa \cite{Oz} to give a proof of a result which is slightly
stronger than Theorem \ref{thm:taut-implies-dupin}.
As we noted in Section \ref{sec:dupin-submanifolds} (see also \cite[p. 33]{CR8}), 
if a curvature surface $S$ has dimension greater
than one, then the corresponding principal curvature is always
constant along $S$, even without the assumption of tautness.  Thus,
Pinkall's proof of Theorem \ref{thm:taut-implies-dupin}
consisted in showing that tautness implies that along any 1-dimensional curvature surface (line of curvature),
the corresponding principal curvature is constant.

Note that we are using
Pinkall's definition of Dupin (Definition \ref{def:dupin}), which does not require that for
every principal space $T_{\mu}$ at every point $x \in M$, there exists
a curvature surface $S$ through $x$ whose tangent space at $x$
is $T_{\mu}$.  However, using the following result of
Ozawa \cite{Oz}, we can show that tautness implies that this
property does hold on $M$ (see Theorem \ref{thm:taut-implies-semi-Dupin}).  

We first state Ozawa's result
and then use it to derive this consequence.  
Ozawa proved his result using Morse-Bott critical point theory (see \cite{BS})
and a careful analysis of the critical submanifolds, and we refer the reader to Ozawa's paper for a complete proof.

\begin{theorem}{(Ozawa)}
\label{thm:Ozawa}
Let $M$ be a taut compact, connected submanifold of ${\bf R}^n$, 
and let $L_p$ be a Euclidean distance function on $M$.
Let $x \in M$  be a critical point of $L_p$, and let $S$ be the
connected component of the critical set of $L_p$ which contains
$x$.  Then $S$ is
\begin{enumerate}
\item[${\rm(a)}$] a smooth compact manifold of dimension equal to the nullity
of the Hessian of $L_p$ at the critical point $x$,
\item[${\rm(b)}$] non-degenerate as a critical manifold,
\item[${\rm(c)}$] taut in ${\bf R}^n$.
\end{enumerate}
\end{theorem}

\begin{remark}{Critical sets of distance functions on taut embeddings}
\label{ozawa-critical-set}

\noindent
{\rm Part (a) of Theorem \ref{thm:Ozawa} implies that for each $p \in {\bf R}^n$, the critical set of $L_p$ is a union of smooth, compact submanifolds of ${\bf R}^n$.
Note that the critical set of $L_p$ is the pre-image of $p$ under the normal exponential map of the submanifold $M$.  Thus, part (a) of Theorem \ref{thm:Ozawa} implies that
for each $p \in {\bf R}^n$, the pre-image of $p$ under the normal exponential map is a union of submanifolds.}
\end{remark}

Using Ozawa's Theorem \ref{thm:Ozawa}, we can prove the following result.

\begin{theorem}
\label{thm:taut-implies-semi-Dupin}
Let $M$ be a taut compact, connected submanifold of ${\bf R}^n$.
Then
\begin{enumerate}
\item[${\rm(a)}$] $M$ is a Dupin submanifold.
\item[${\rm(b)}$]  Given a principal space $T_{\lambda}$ of a shape operator
$A_{\xi}$ at a point $x \in M$, there exists a curvature
surface $S$ through $x$ whose tangent space at $x$ is equal to
$T_{\lambda}$, and $\lambda$ is constant along $S$.
\end{enumerate}
\end{theorem}

\begin{proof}
Note that part (b) implies part (a), so we will prove part (b).
Let $f:M \rightarrow {\bf R}^n$ be a taut embedding.
Let $\xi$ be a unit normal vector at an arbitrary
point $x \in M$, and let $\lambda$ be a principal curvature of
$A_{\xi}$.  We first consider the case where $\lambda \neq 0$. Let $p = f(x) + (1/ \lambda)\xi$ be the focal point
of $(M,x)$ determined by the principal curvature $\lambda$ of $A_{\xi}$.
Then the distance function $L_p$ has a degenerate critical point at $x$
and the nullity of the Hessian of $L_p$ at $x$ is equal to the
multiplicity $m$ of $\lambda$ as an eigenvalue of $A_{\xi}$ 
(see \cite[p. 36]{Mil}).  By Ozawa's Theorem \ref{thm:Ozawa}, the connected
component $S$ of the critical set of $L_p$ containing $x$ is
a smooth submanifold (a critical submanifold) of dimension $m$.  We will now show that $S$
is the desired curvature surface and that the corresponding
principal curvature is constant along $S$.

The function $L_p$ has a constant value, which is $1/{\lambda}^2$,
on the critical submanifold $S$.  Thus, for every point $y \in S$, 
the vector $p-f(y)$ is normal to $f(M)$ at $f(y)$, and it has
length $1/|\lambda|$.  So we can extend the normal vector $\xi$ to
a unit normal vector field to $f(M)$ along $S$,
which we also denote by $\xi$, by setting $\xi(y) = \lambda (p-f(y))$.
Note that $p$ is a focal point of $(M,y)$ for every point
$y \in S$, and Ozawa's theorem implies that the number $\lambda$
is an eigenvalue of $A_{\xi(y)}$ of multiplicity
$m$ = dim $S$ for every point $y \in S$.  Thus, the principal
curvature $\lambda$ is constant along $S$.  We next show that $T_yS$ equals the 
principal space $T_{\lambda}(y)$ at each point $y \in S$, and that 
the normal field $\xi$ is parallel along $S$ with respect to the
normal connection.  Consider the focal map,
\begin{displaymath}
f_{\lambda}(y) = f(y) + \frac{1}{\lambda} \xi(y),
\end{displaymath}
for $y \in S$.  Then $f_{\lambda}(y) = p$ for all $y \in S$.
Let $X$ be any tangent vector to $S$ at any point $y \in S$.
Then $(f_{\lambda})_*X = 0$, since $f_{\lambda}$ is constant on $S$.
On the other hand,
\begin{displaymath}
(f_{\lambda})_*X = f_*X + \frac{1}{\lambda} \xi_*X,
\end{displaymath}
and $\xi_*X = D_X\xi = f_*(-A_{\xi}X) + \nabla_X^{\perp}\xi$.  
Therefore,
\begin{displaymath}
(f_{\lambda})_*X = f_*(X -\frac{1}{\lambda} A_{\xi}X) + 
\frac{1}{\lambda} \nabla_X^{\perp}\xi.
\end{displaymath}
Since $(f_{\lambda})_*X = 0$, we see that $A_{\xi}X = \lambda X$ and
$\nabla_X^{\perp}\xi = 0$.  Thus, $\xi$ is parallel along $S$
and $T_yS \subset T_{\lambda}(y)$.  Since $T_yS$ and $T_{\lambda}(y)$
have the same dimension, they are equal.  So $S$ is the
curvature surface through $y$ corresponding to the principal curvature $\lambda$, which is constant along $S$.

Now suppose that $\lambda =0$ is an eigenvalue of $A_{\xi}$ at $x$.  Let 
\begin{displaymath}
\sigma:{\bf R}^{n} - \{ q \} \rightarrow {\bf R}^{n} - \{ q \}
\end{displaymath}
be an inversion in a sphere
centered at a point $q \in {\bf R}^n$ chosen so that $q \notin f(M)$, and so that the principal curvature $\mu$
of the embedding $\sigma f:M \rightarrow {\bf R}^n$ corresponding to $\lambda$ is not zero
(see \cite[pp. 20--22]{CR8}).  
Since the conformal transformation $\sigma$ preserves tautness, the transformation
$\sigma f$ is taut and the principal curvature $\mu \neq 0$, so the argument above shows that there
exists a curvature surface $V$ of $\sigma f$ through $x$ whose tangent space at $x$ is equal to $T_{\mu}$, and $\mu$ is constant along $V$.
Applying the inversion $\sigma$ again, we get a curvature surface $S = \sigma (V)$ corresponding to the principal curvature $\lambda$ of $f = \sigma^2 f$, 
and $\lambda$ is constant along $S$, as needed in part (b) of the theorem. This completes the proof.
\end{proof}

\begin{remark}{On the relationship between taut and ``semi-Dupin''}
\label{rem:semi-Dupin}

\noindent
{\rm In the book of Cecil and Ryan \cite[p. 189]{CR7}, a Dupin (but not necessarily proper Dupin) hypersurface 
which satisfies Condition (b) in Theorem \ref{thm:taut-implies-semi-Dupin} was called ``semi-Dupin.''  Conjecture 6.19 \cite[p. 190]{CR7}
 in that same book stated that the notions of taut and semi-Dupin are equivalent for compact embedded
 hypersurfaces in ${\bf R}^n$.
Theorem \ref{thm:taut-implies-semi-Dupin}
gives an affirmative answer in one direction to Conjecture 6.19, that is, taut implies semi-Dupin for compact, connected submanifolds of ${\bf R}^n$.  In a recent preprint, Wiesendorf \cite{Wiesendorf1} gives a proof that semi-Dupin implies 
taut for compact, connected hypersurfaces embedded in ${\bf R}^n$.}
\end{remark}

\begin{remark}{Taut embeddings into complete Riemannian manifolds}
\label{rem:TTW}

\noindent
{\rm Using different approaches, Grove and Halperin \cite{GH-91}, and independently, Terng and Thorbergsson 
\cite{TTh1}, extended the notion of tautness to properly embedded
submanifolds of complete Riemannian manifolds.  
Specifically, a submanifold $M$ of a complete Riemannian manifold $N$ is said to be taut if there exists a field ${\bf F}$ such that each energy functional,
\begin{equation}
\label{eq:energy}
E_p (\gamma) = \int_{0}^{1} |\gamma\ '(t)|^2 dt,
\end{equation}
on the space $\mathcal P (N,M \times p)$ of $H^1$-paths $\gamma :[0,1] \rightarrow N$ from $M$ to a fixed point $p \in N$ is a perfect Morse function with respect to ${\bf F}$, if $p$ is not a
focal point of $M$. (Here a path is $H^1$ if it is absolutely continuous and the length of its derivative is square integrable.)

This definition can be shown to agree with the usual definition of tautness for submanifolds of Euclidean space.
Terng and Thorbergsson \cite{TTh1} showed that many of the important properties of taut embeddings
into Euclidean space have natural analogues in this more general setting. 

In a subsequent paper, Wiesendorf \cite{Wiesendorf} 
proved that a compact, connected submanifold $M$ embedded in a complete Riemannian manifold $N$ is taut if and only if for each point $p$ in $N$, the pre-image of $p$ under the normal exponential
map  of $M$ is a union of submanifolds.  As noted in Remark \ref{ozawa-critical-set} above,
part (a) of Ozawa's Theorem \ref{thm:Ozawa} implies that
for a taut compact, connected submanifold of ${\bf R}^n$,
the pre-image of $p$ under the normal exponential
map  of $M$ is a union of submanifolds. 

Wiesendorf  \cite{Wiesendorf}
also proved that if $M$ is taut with respect to any field ${\bf F}$, then $M$ is also taut with respect to 
${\bf Z}_2$.  In addition, Wiesendorf proved several results concerning singular Riemannian foliations, 
all of whose leaves are taut (see also Lytchak \cite{Lytchak-09}--\cite{Lytchak-10},  and
Lytchak and Thorbergsson \cite{LT}--\cite{LT-10}).

In the context of taut submanifolds of complete Riemannian manifolds, Taylor \cite{Taylor-2000} gave a classification of immersions of $S^{n-1}$ into a complete
Riemannian manifold $N^n$ which have odd order in homotopy and are taut. (See also Hebda \cite{Heb1}--\cite{Heb3},  Kahn \cite{Kahn},
and Ruberman \cite{Ruberman} for related results.)} 
\end{remark}

\section{Submanifolds in Lie Sphere Geometry}
\label{sec:submanifolds-lie-geometry}
 
In this section,
we give a brief description of the method for studying submanifolds
of Euclidean space ${\bf R}^n$ and the sphere $S^n$ using Lie sphere geometry (see
\cite{Cec9}, \cite{CC1} or \cite{P2} for more detail).  
As we noted
earlier, the Dupin property is invariant under stereographic
projection between ${\bf R}^n$ and $S^n$ (see \cite[pp. 147--148]{CR7}).
At times, it is simpler to work in $S^n$, and we will give our
description in those terms here. 

Let ${\bf R}_2^{n+3}$ be a real vector space of
dimension $n+3$ endowed with a metric of signature $(n+1,2)$,
\begin{equation}
\label{eq:lie-metric}
\langle x,y \rangle = -x_1 y_1 + x_2 y_2 + \cdots +x_{n+2} y_{n+2} -
x_{n+3} y_{n+3}.
\end{equation}
Let $e_1, \ldots ,e_{n+3}$ denote the standard orthonormal
basis with respect to this metric, with $e_1$ and $e_{n+3}$
timelike.  Let ${\bf RP}^{n+2}$ be the real projective space of lines through
the origin in ${\bf R}_2^{n+3}$,  and let $Q^{n+1}$ be the quadric
hypersurface of ${\bf RP}^{n+2}$ determined by the equation $\langle x,x \rangle = 0$.  This
hypersurface is called the {\em Lie quadric}.
We consider $S^n$ to be the unit sphere in the Euclidean space
${\bf R}^{n+1}$ spanned by the vectors $e_2, \ldots , e_{n+2}$.

The points in $Q^{n+1}$ are in bijective correspondence
with the set of all oriented hyperspheres and point
spheres in $S^n$.  Specifically,
the oriented hypersphere $S$ with center $p \in S^n$ and
signed radius $\rho$ corresponds to the point
\begin{equation}
\label{eq:lie-sphere-in-sphere}
[(\cos \rho , p, \sin \rho )]
\end{equation} in $Q^{n+1}$, where the 
square brackets denote the point in ${\bf RP}^{n+2}$ given by the
homogeneous coordinates within the round brackets.  The point
spheres in $S^n$ correspond to those points with $\rho = 0$.

The orientation of a sphere $S$ in $S^n$ is
determined by a choice of unit normal field to $S$ in $S^n$.  Geometrically, we take the positive radius $\rho$
in \eqref{eq:lie-sphere-in-sphere} 
to correspond to the field of unit normals which are tangent vectors to geodesics from $-p$ to $p$,
and the negative radius corresponds to the opposite orientation.
Each oriented sphere can be considered in two ways, with center $p$ and signed radius $\rho, - \pi < \rho < \pi$,
or with center $-p$ and the appropriate signed radius $\rho \pm \pi$.  Point spheres do not have an orientation.

Due to the signature of the metric $\langle \ ,\ \rangle$ given in equation \eqref{eq:lie-metric},
the Lie quadric contains projective lines, but no linear
subspaces of ${\bf RP}^{n+2}$ of higher dimension.  The line
$[x,y]$ determined by two points $[x]$ and $[y]$ of $Q^{n+1}$
lies on the quadric if and only if $\langle x,y \rangle =0$.  In terms of the
geometry of $S^n$,
this means that the two hyperspheres corresponding
to $[x]$ and $[y]$ are in {\em oriented contact}, i.e., the two hyperspheres are tangent to each other
and their orientations agree at the point of contact. 

The points
on a line on the quadric correspond to the parabolic pencil of oriented
hyperspheres in $S^n$ that are in oriented contact at a point $(p,\xi )$ in the
unit tangent bundle $T_1 S^n$ to $S^n$, where $\xi$ is a unit
tangent vector to $S^n$ at the point $p \in S^n$. 

Here, we can represent
$T_1S^n$ as the $(2n-1)$-dimensional submanifold of 
$S^n \times S^n \subset{\bf R}^{n+1} \times {\bf R}^{n+1}$ given by
\begin{equation}
\label{eq:3.1.1}
T_1S^n = \{(p, \xi) \mid \quad |p| =1, \quad |\xi| = 1, \quad p \cdot \xi = 0\},
\end{equation}
where $ p \cdot \xi$ denotes the Euclidean inner product of $p$ and $\xi$ in ${\bf R}^{n+1}$.

This leads to a
natural diffeomorphism from $T_1 S^n$ to the manifold
$\Lambda ^{2n-1}$ of projective lines on $Q^{n+1}$ given by
$(p,\xi) \mapsto[k_1, k_2]$, where 
\begin{displaymath}
k_1 = (1,p,0), \quad k_2 = (0,\xi,1).  
\end{displaymath}
In terms of the geometry of spheres in
$S^n$, the point $k_1$ corresponds to the point sphere in the parabolic
pencil, and $k_2$ corresponds to the great sphere in the pencil.
We will refer to 
the elements of $T_1 S^n$ as {\em contact elements}.

A {\em Lie sphere transformation} is a projective
transformation of ${\bf RP}^{n+2}$ which maps $Q^{n+1}$ to
itself.  In terms of the geometry of $S^n$, a Lie sphere
transformation maps oriented hyperspheres to
oriented hyperspheres.  Furthermore, a Lie
sphere transformation preserves oriented contact
of spheres, since it takes lines on $Q^{n+1}$ to lines on
$Q^{n+1}$.  

The group of Lie sphere transformations is
isomorphic to $O(n+1,2)/ \{ \pm I \}$, where $O(n+1,2)$
is the orthogonal group for the metric in \eqref{eq:lie-metric}.
A {\em M\"{o}bius transformation} is a Lie sphere transformation
which takes point spheres to point spheres.  As a transformation
on $S^n$ itself, a M\"{o}bius transformation is conformal.

For the spherical metric, there is a parallel transformation $P_t$ that adds $t$ to the signed 
radius of each sphere while keeping the center fixed.  As we saw in \eqref{eq:lie-sphere-in-sphere}, 
the sphere in $S^n$ with center
$p$ and signed radius $\rho$ is represented by the point $[(\cos \rho, p, \sin \rho)]$ in $Q^{n+1}$.
One easily checks that 
{\em spherical parallel transformation}
$P_t$ is accomplished by the 
transformation in $O(n+1,2)$ given by:
\begin{eqnarray}
\label{eq:2.5.3}
P_t e_1 & = & \cos t \ e_1 + \sin t \ e_{n+3}, \nonumber \\
P_t e_{n+3} & = & - \sin t \ e_1 + \cos t \ e_{n+3},\\
P_t e_i & = & e_i, \quad 2 \leq i \leq n+2. \nonumber
\end{eqnarray}
The Lie
sphere group is generated by M\"{o}bius transformations and 
parallel transformations $P_t$.  (See \cite[pp. 25--49]{Cec9}.)

The manifold  $\Lambda ^{2n-1}$ has a {\em contact structure},
i.e., a globally defined 1-form $\omega$ such that
$\omega \wedge d\omega ^{n-1}$ never vanishes on
$\Lambda ^{2n-1}$.  The condition $\omega = 0$ defines a
codimension one distribution $D$ on $\Lambda ^{2n-1}$, which has
integral submanifolds of dimension $n-1$, but none of higher
dimension.  A {\em Legendre submanifold} is one of these integral
submanifolds of maximal dimension, i.e.,
an immersion $\lambda : M^{n-1} \rightarrow
\Lambda ^{2n-1}$ such that $\lambda ^* \omega = 0$. (See \cite[pp. 51--64]{Cec9}.)\\

A Legendre submanifold is determined
by two functions $k_1 , k_2 $ from an $(n-1)$-dimensional
manifold $M$ to
${\bf R}^{n+3}_2$ satisfying the conditions:
\begin{enumerate}
\item[(L1)] For all $x \in M$ , the vectors $k_1 (x)$ and $k_2 (x)$
are linearly independent and 
$\langle k_i (x) , k_j (x) \rangle  = 0$, for $i,j = 1,2$.
\item[(L2)] There is no non-zero $X \in T_x M$, for any $x \in M$,
such that $dk_1 (X)$ and $dk_2 (X)$ are both in Span$\{ k_1 (x),
k_2 (x) \} $.
\item[(L3)] $\langle dk_1 (X), k_2 (x) \rangle = 0$, for all $X \in T_x M$,
for all $x \in M$.\\
\end{enumerate}
The Legendre submanifold is then defined by 
$\lambda (x) = [k_1 (x) , k_2 (x)]$.  Conditions (L1)--(L3) are
preserved if one reparametrizes by taking $\tilde{k} _1 =
\alpha k_1 + \beta k_2 $ and $\tilde{k} _2 = \gamma k_1 +
\delta k_2 $, where $\alpha , \beta , \gamma , \delta$ are smooth
real-valued functions on $M$ with $\alpha \delta -
\beta \gamma$ never equal to zero.

Condition (L1) means that $k_1$ and $k_2$ determine a line on the
quadric for each $x \in M$.  Condition (L2) means that 
$\lambda$ is an immersion, and Condition (L3) means that 
$\lambda ^* \omega = 0$. 

An immersion $f:M^{n-1} \rightarrow S^n$ with field of unit
normals $\xi :M^{n-1} \rightarrow S^n$ naturally induces
a Legendre submanifold $\lambda = [k_1, k_2],$ where 
\begin{equation}
\label{eq:legendre-parametrization}
k_1 = (1,f,0), \quad k_2 = (0,\xi ,1).
\end{equation}
For each $x \in M^{n-1}, [k_1(x)]$ is the point sphere in the
pencil of spheres in $S^n$ corresponding to $\lambda (x)$, and
$[k_2 (x)]$ is the great sphere in the pencil.  The Legendre submanifold  $\lambda$ is called the
{\em Legendre lift} of the oriented hypersurface $f$ with field of unit normals $\xi$.

An immersed submanifold $\phi :V \rightarrow S^n$ of codimension 
greater than one also induces a Legendre submanifold 
whose domain is the bundle $B^{n-1}$ of unit normal
vectors to $\phi (V)$.  The unit normal bundle $B^{n-1}$ can be considered to be the
submanifold of $V \times S^n$ given by
\begin{equation}
\label{eq:unit-normal-bundle}
B^{n-1} = \{ (x,\xi)\ | \ \phi(x) \cdot \xi = 0, \ d\phi(X) \cdot \xi = 0,\  \mbox{\rm for all }X \in T_xV\}.
\end{equation} 
The {\em Legendre lift of the submanifold}
$\phi$ is the map $\lambda: B^{n-1} \rightarrow \Lambda^{2n-1}$ defined by
\begin{equation}
\label{eq:3.3.2}
\lambda( x,\xi) = [k_1(x,\xi), k_2(x,\xi)],
\end{equation} 
where
\begin{equation}
\label{eq:3.3.3}
k_1(x,\xi) = (1, \phi(x), 0), \quad k_2(x,\xi) = (0, \xi, 1).
\end{equation} 
Geometrically, $\lambda(x,\xi)$ is the line on the quadric $Q^{n+1}$ corresponding to
the contact element $(\phi (x), \xi ) \in T_1S^n$. 
In this case,
the point sphere map, 
\begin{displaymath}
k_1 (x, \xi ) = (1, \phi (x), 0),
\end{displaymath}
has constant rank equal to the dimension of $V$.   Note that 
for a general Legendre submanifold $\lambda$,
the point sphere map does not have constant rank.    

A Lie sphere transformation $\beta$ maps lines on
$Q^{n+1}$ to lines on $Q^{n+1}$, so it naturally
induces a map $\tilde{\beta }$ from 
$\Lambda ^{2n-1}$ to itself.  If $\lambda$ is a
Legendre submanifold, then $\tilde{\beta }\lambda $
is also a Legendre submanifold, which is denoted
$\beta \lambda$ for short.  

These two Legendre
submanifolds, $\lambda$ and $\beta \lambda$, are said to be {\em Lie equivalent}.  
If $\beta$ is a M\"{o}bius transformation, then the two Legendre
submanifolds are said to be {\em M\"{o}bius equivalent}.
Finally, if $\beta$ is the parallel transformation $P_t$ and
$\lambda$ is the Legendre lift of an
oriented hypersurface $f:M \rightarrow S^n$ with field of unit normals $\xi$, then
$P_t\lambda$ is the Legendre lift of the
parallel hypersurface $f_{-t}$ (see \cite[p. 67]{Cec9}).

Suppose that $\lambda = [k_1 , k_2]$ is a Legendre submanifold.  Let
$x \in M$ and let $r$ and $s$ be real numbers at least one
of which is non-zero.  The sphere corresponding to the point
\begin{equation}
\label{eq:curvature-sphere}
[K] = [rk_1 (x) + sk_2 (x) ]
\end{equation}
is called a {\em curvature sphere} of $\lambda$ at $x$, if there
exists a non-zero $X \in T_x M$ such that
\begin{equation}
\label{eq:curv-sphere-derivative}
r dk_1 (X) + s dk_2 (X) \in {\rm Span} \{ k_1 (x),
k_2 (x) \} .
\end{equation}
This definition is invariant under a reparametrization of $\lambda$
by a pair $\{ \tilde{k} _1 , \tilde{k} _2 \}$, where
$\tilde{k} _1 = \alpha k_1 + \beta k_2 $ and $\tilde{k} _2 = \gamma k_1 +
\delta k_2 $, where $\alpha , \beta , \gamma , \delta$ are smooth
real-valued functions on $M$ with $\alpha \delta -
\beta \gamma$ never equal to zero.

To see the relationship between curvature spheres and principal
curvatures, suppose now that
$\lambda = [k_1 , k_2 ]$ as in \eqref{eq:legendre-parametrization}.  At a given
$x \in M$, we can write the distinct curvature spheres
in the form
\begin{equation}
\label{eq:distinct-curvature-spheres}
[K_i ] = [\mu _i k_1 + k_2 ], \quad 1\leq i \leq g .
\end{equation}
In the case where the map $f$ in \eqref{eq:legendre-parametrization} is an immersion, these
$\mu _i$ are the principal curvatures
of $f$ at $x$.  In terms of the geometry of $S^n$, the curvature
sphere at $x$ corresponding to a principal curvature $\mu _i$
is the oriented hypersphere in oriented contact with $f(M)$
at $f(x)$ and centered at the focal point determined by the
principal curvature $\mu _i$. (See \cite[pp. 64--82]{Cec9} for more detail.)

We refer to the $\mu_i$  in \eqref{eq:distinct-curvature-spheres} as the {\em principal
curvatures} of $\lambda$.
These principal curvatures are not Lie
invariant, and they depend on the special parametrization \eqref{eq:legendre-parametrization}
for $\lambda$.  However, Miyaoka \cite{Mi2} proved that the
cross-ratio of any four of 
these principal curvatures is Lie invariant.

These cross-ratios are known as the {\em Lie curvatures} of $\lambda$, and they were 
discovered by Miyaoka \cite{Mi2}.
Such a cross-ratio is Lie invariant because it is equal to
the cross-ratio of the corresponding four curvature spheres
on the line $\lambda (x)$. Since a Lie sphere transformation
$\beta$ is a projective transformation, and it
maps the curvature spheres of $\lambda$ to the
curvature spheres of $\beta \lambda$, it preserves these
cross-ratios (see also, \cite[pp. 72--82]{Cec9}).

A Legendre submanifold $\lambda : M^{n-1} \rightarrow \Lambda ^{2n-1}$ is said to be
{\em Dupin} if along each curvature surface, the corresponding curvature
sphere map is constant, and a Dupin submanifold is said to be {\em proper Dupin}, if the number of
distinct curvature spheres is constant on $ M^{n-1}$.  Pinkall \cite{P2} showed
that both of these concepts are invariant under Lie sphere
transformations.  In the case where a Legendre submanifold is 
the Legendre lift of a submanifold of $S^n$ or  ${\bf R}^n$, these definitions agree with
those given in Section \ref{sec:dupin-submanifolds}.

A Dupin submanifold $\lambda : M^{n-1} \rightarrow \Lambda ^{2n-1}$ is called {\em reducible} if
it is locally Lie equivalent to a  Dupin submanifold obtained as a 
result of one of Pinkall's four constructions in Definition \ref{def:pinkalls-constructions}
(suitably generalized to the context of Lie sphere geometry,
see \cite[pp. 127--148]{Cec9}, \cite{CecGD}), 
and a Dupin submanifold of $S^n$ or  ${\bf R}^n$ is called reducible, if its Legendre
lift is reducible.  A Dupin submanifold that is not reducible is called {\em irreducible}.

Pinkall \cite{P2} found a useful formulation for reducibility
in terms of Lie sphere geometry as follows (see also \cite[pp. 127--148]{Cec9}).
\begin{theorem}
\label{thm:pinkall-reducibility}
Let $\lambda : M^{n-1} \rightarrow \Lambda ^{2n-1}$ be a proper
Dupin submanifold with distinct curvature spheres $K_1,
\ldots K_g$.  Then $\lambda$ is reducible if and only if
for some $i, 1 \leq i \leq g$,  the image of the
curvature sphere map $K_i$ is contained in an
$n$-dimensional linear subspace of ${\bf RP}^{n+2}$.
\end{theorem}

One can obtain a reducible compact proper Dupin hypersurface in ${\bf R}^3$ with two principal curvatures 
by revolving a circle $C$ in ${\bf R}^3$ about an axis ${\bf R}^1 \subset {\bf R}^3$ that is disjoint from $C$ 
to obtain a torus of revolution.
However, as pointed out in Remark \ref{rem:pinkall-constructions-compactness},
there are often problems in trying to construct a compact proper Dupin hypersurface
by applying one of Pinkall's constructions to a lower dimensional Dupin submanifold.  In fact,
Cecil, Chi and Jensen \cite{CCJ2} (see also \cite[pp. 146--147]{Cec9}) proved the following result.

\begin{theorem}
\label{thm:CCJ-2007}
(Cecil-Chi-Jensen, 2007)
If $M^{n-1} \subset {\bf R}^n$ is a compact, connected proper Dupin hypersurface with $g \geq 3$
principal curvatures, then $M^{n-1}$ is irreducible.
\end{theorem}

The proof is accomplished by comparing the sum of the  ${\bf Z}_2$-Betti numbers 
of a compact proper Dupin hypersurface, as found in Theorem \ref{thm:dupin-restriction-g}, 
with the sum of
the  ${\bf Z}_2$-Betti numbers of a compact proper Dupin hypersurface obtained by one of Pinkall's constructions.

\begin{remark}{Using irreducibility to study compact Dupin hypersurfaces}
\label{rem:irreducibility}

\noindent
{\rm From Theorem \ref{thm:CCJ-2007}, we see that 
one approach to obtaining classifications of compact proper Dupin hypersurfaces with more
than two principal curvatures is to assume that the hypersurface is irreducible, and 
then work locally in the context of Lie sphere geometry, using the method of moving frames.  This approach has been used successfully in the papers of 
Pinkall \cite{P1}, \cite{P3}, Cecil and Chern \cite{CC2}, Cecil and Jensen \cite{CJ2}--\cite{CJ3}, 
and Cecil, Chi and Jensen \cite{CCJ2}, and  we will discuss it in more detail in Section \ref{sec:compact-proper-dupin}
(see also \cite[pp. 168--190]{Cec9}).}
\end{remark}

Another important question in classifying proper Dupin
submanifolds is whether the Dupin submanifold is Lie equivalent to the Legendre lift of
an isoparametric hypersurface in $S^n$.  This condition
also has a natural formulation in Lie sphere geometry given in Theorem \ref{thm:lie-eq-isop-hyp} below
(see Cecil  \cite{CecKod} or \cite[p. 77]{Cec9}).

Recall
that a line in ${\bf RP}^{n+2}$ is said to be {\em timelike} if
it contains only timelike points.  This means that 
an orthonormal basis for the 2-plane in  ${\bf R}_2^{n+3}$,
determined by the timelike line, consists of two timelike
vectors.  An example is the line $[e_1,e_{n+3}]$.

\begin{theorem}
\label{thm:lie-eq-isop-hyp}
Let $\lambda : M^{n-1} \rightarrow \Lambda ^{2n-1}$ be
a Legendre submanifold with $g$ distinct curvature spheres
$K_1,\ldots,K_g$ at each point.  Then $\lambda$ is
Lie equivalent to the Legendre lift of an
isoparametric hypersurface in $S^n$ if and only if there
exist $g$ points $P_1,\ldots,P_g$ on a timelike line in
${\bf RP}^{n+2}$ such that 
\begin{equation}
\label{eq:lie-geometric-criterion-isop-hyp}
\langle K_i,P_i \rangle = 0, \quad 1 \leq i \leq g.
\end{equation}
\end{theorem}

For more general considerations of Lie contact structures on manifolds, see Miyaoka \cite{Mi4}--\cite{Mi7}.

\section{Compact Proper Dupin Submanifolds}
\label{sec:compact-proper-dupin}

In this section, we consider compact proper Dupin
submanifolds embedded in the sphere $S^n$ or in Euclidean space ${\bf R}^n$.  Since a tube
of sufficiently small radius $\epsilon$ over a compact
proper Dupin submanifold of codimension greater than one is
a compact proper Dupin hypersurface, we will 
for the most part restrict our
attention to the codimension one case. (See also \cite{Cec-compact-dupin}, \cite{Cec-classifications-dupin},
\cite[pp. 308--322]{CR8} 
for further descriptions of the results contained in this section.)

Let $M$ be a compact
proper Dupin hypersurface embedded in $S^n$ with $g$ distinct
principal curvatures.  
As noted in Theorem \ref{thm:dupin-restriction-g}, the number $g$
must be $1,2,3,4$ or 6.  Of course, in the case $g=1$, the
hypersurface $M$ is totally umbilic and must be a great or small
hypersphere in $S^n$.  The case $g=2$ was handled in 1978
by Cecil and Ryan \cite{CRMA} who proved the following.
\begin{theorem}
\label{thm:C-R-math-ann}
A compact, connected proper Dupin hypersurface embedded in $S^n$
with two distinct principal curvatures is M\"{o}bius equivalent to
an isoparametric hypersurface, i.e., a standard product of two spheres.
\end{theorem}
Next Miyaoka \cite{Mi1} handled the case $g=3$, where the full
Lie sphere group was needed to get equivalence with an
isoparametric hypersurface.
\begin{theorem}
\label{thm:miyaoka-1}
A compact, connected proper Dupin hypersurface embedded in $S^n$
with three distinct principal curvatures is Lie equivalent
to an isoparametric hypersurface.
\end{theorem}

Thus, at the time of the writing of the book \cite{CR7} of Cecil and Ryan in 1985, it was known that every compact, connected proper Dupin hypersurface 
$M \subset S^n$ (or ${\bf R}^n$) with
$g= 1,2$ or 3 principal curvatures is Lie equivalent to an isoparametric hypersurface in $S^n$.
At that time, every other known example
of a compact, connected proper Dupin hypersurface in $S^n$ 
was also Lie equivalent to an isoparametric hypersurface in $S^n$. 
This led to the following
conjecture by Cecil and Ryan \cite[p. 184]{CR7} (which we have rephrased slightly).

\begin{conjecture} 
\label{cecil-ryan} 
(Cecil-Ryan, 1985) Every compact, connected proper Dupin hypersurface $M \subset S^n$ $($or ${\bf R}^n)$
is Lie equivalent to an isoparametric hypersurface in $S^n$.
\end{conjecture}

In the case $g=4$, the conjecture remained unsolved
for several years until finally, in papers published in 1989, 
counterexamples to the conjecture were constructed 
independently by Pinkall and Thorbergsson \cite{PT}
and Miyaoka and Ozawa \cite{MO}. These two constructions lead to different types of counterexamples
to the conjecture, and
the method of Miyaoka and Ozawa also 
yields counterexamples to the conjecture with $g=6$ principal curvatures.  

In both constructions, a fundamental
Lie invariant, the Lie curvature, which was discovered by Miyaoka \cite{Mi2},
was used to show that the examples are not Lie equivalent to
an isoparametric hypersurface.  Specifically, if $M$ is a proper Dupin hypersurface 
with four principal curvatures,
then the {\em Lie curvature} $\psi$ is defined to be
the cross-ratio of these principal curvatures.  If $M$ has six principal curvatures, then the Lie curvatures are the cross-ratios of the principal curvatures taken four at a time.

Viewed in the context of projective geometry, at each point $x \in M$, a Lie curvature
is the cross-ratio of four points along a projective line lying on $Q^{n+1}$ corresponding
to four curvature spheres of $M$ at $x$.  
Since a Lie sphere transformation maps curvature spheres to curvature
spheres and preserves cross-ratios, a Lie curvature is invariant under Lie sphere transformations
(see  \cite[pp. 72--82]{Cec9} for more detail).  

From the work of M\"{u}nzner \cite{Mu}--\cite{Mu2}, 
it is easy to show that in the case $g=4$, the 
Lie curvature $\psi$ has the constant value $1/2$ 
on an isoparametric hypersurface  (when the principal curvatures are appropriately ordered). For
the counterexamples to the conjecture, it was shown that $\psi \neq 1/2$ at some points, and therefore the examples cannot be Lie equivalent to an isoparametric hypersurface.   In fact, it can be shown that the Lie curvature $\psi$
is not constant on the counterexamples to the conjecture.

The examples of Pinkall and Thorbergsson  are obtained by 
taking certain deformations of the isoparametric hypersurfaces
of OT-FKM-type constructed by  Ozeki and Takeuchi \cite{OT}--\cite{OT2}, and by
Ferus, Karcher and M\"{u}nzner \cite{FKM}, using representations of Clifford algebras.  
Pinkall and Thorbergsson proved that their examples are not Lie
equivalent to an isoparametric hypersurface by showing that the Lie curvature does not have the
constant value $\psi = 1/2$, 
as required for a hypersurface with $g=4$ that is Lie equivalent 
to an isoparametric hypersurface.  Using their methods, one can also show directly
that the Lie curvature is not constant on their examples (see, for example,  \cite[pp. 309--314]{CR8}).

The construction of counterexamples to  Conjecture \ref{cecil-ryan} due to Miyaoka and Ozawa \cite{MO} (see also
\cite[pp. 117--123]{Cec9}) is based on the Hopf fibration
$h:S^7 \rightarrow S^4$.  
Miyaoka and Ozawa first showed that if $W^3$ is a taut
compact submanifold of $S^4$, then $M = h^{-1}(W^3)$ is a taut compact submanifold of
$S^7$.  

Using this and the fact that tautness is equivalent to proper Dupin
for a compact, connected hypersurface in $S^n$ on which the number $g$ of distinct principal 
curvatures is constant (Theorem \ref{thm:taut-dupin}),
they show that if $W^3$ is a proper Dupin hypersurface in
$S^4$ with $g$ distinct principal curvatures, then
$h^{-1}(W^3)$ is a proper Dupin hypersurface in $S^7$ with $2g$ principal
curvatures.  

To complete the argument, they show that if a compact, connected 
hypersurface $W^3 \subset S^4$ is proper Dupin
but not isoparametric, then the Lie curvatures of
$h^{-1}(W^3)$ are not constant, and therefore $h^{-1}(W^3)$ is not Lie
equivalent to an isoparametric hypersurface in $S^7$.  For $g=2$ or 3, this gives
a compact proper Dupin hypersurface $M = h^{-1}(W^3)$ in $S^7$ with $g=4$ or 6, respectively,
that is not Lie equivalent to an isoparametric hypersurface.  (See also \cite[pp. 112--123]{Cec9}
or the paper \cite{Cec11} for a detailed description
of the examples of Pinkall-Thorbergsson and Miyaoka-Ozawa.)

As we have seen, all of the hypersurfaces described above are shown to be 
counterexamples to Conjecture \ref{cecil-ryan} by proving that they do not have
constant Lie curvatures.  This led to a revision of Conjecture \ref{cecil-ryan} by Cecil, Chi and Jensen
\cite[p. 52]{CCJ4} in 2007 that contains the additional assumption of constant Lie curvatures.
This revised conjecture
is still an open problem, although it  has been shown to be true in some cases,
which we will describe after stating the conjecture.

\begin{conjecture}
\label{revised-conjecture}
(Cecil-Chi-Jensen, 2007)
Every compact, connected proper Dupin hypersurface in $S^n$ with four or six principal curvatures
and constant Lie curvatures is Lie equivalent to an isoparametric hypersurface.
\end{conjecture}

\begin{remark}{Miyaoka's work on Conjecture \ref{revised-conjecture}} 
\label{rem:miyaoka-examples}

\noindent
{\rm In 1989, Miyaoka \cite{Mi2}--\cite{Mi3} showed that if some additional assumptions are made regarding the intersections of the leaves of the various principal foliations, then Conjecture \ref{revised-conjecture}
is true in both cases
$g=4$ and 6.  
Thus far, however, it has not been proven that Miyaoka's additional assumptions are satisfied, in general.}
\end{remark}
 
Cecil, Chi and Jensen \cite{CCJ2} 
made progress on Conjecture \ref{revised-conjecture} in the case $g=4$ by using the
fact that compactness implies irreducibility for a proper Dupin hypersurface
with $g \geq 3$ (see Theorem \ref{thm:CCJ-2007}), 
and then working locally with irreducible proper hypersurfaces in the context 
of Lie sphere geometry, as described earlier in Remark \ref{rem:irreducibility}.  This leads to the following 
local version of Conjecture \ref{revised-conjecture}:

\begin{conjecture}
\label{revised-conjecture-local}
(Cecil-Chi-Jensen, 2007)
Every irreducible connected proper Dupin hypersurface in $S^n$ with four or six principal curvatures
and constant Lie curvatures is Lie equivalent to an isoparametric hypersurface.
\end{conjecture}

\begin{remark}{A reducible example with constant Lie curvature $\psi = 1/2$} 
\label{rem:reducible-with-lie curv-1/2}

\noindent
{\rm In \cite{CecKod} (see also \cite[pp. 80--82]{Cec9},  \cite[pp. 304--306]{CR8}),  Cecil constructed an example of
a reducible, non-compact proper Dupin submanifold with $g=4$ distinct principal 
curvatures and constant Lie curvature
$\psi = 1/2$, which is not Lie equivalent to an open subset of an isoparametric
hypersurface with four principal curvatures in $S^n$.  This example cannot be made compact without destroying the property that the number $g$ of distinct curvature spheres equals four at each point. This example shows that the 
hypothesis of irreducibility is necessary in Conjecture \ref{revised-conjecture-local}.}
\end{remark}

We now mention some notable facts pertaining to Conjecture \ref{revised-conjecture-local}.\\

\noindent
Case $g = 4$: 
In this case, there is only one Lie curvature,
\begin{equation}
\label{eq:lie-curv}
\psi = \frac{(\mu_1 -\mu_2)(\mu_4 - \mu_3)}{(\mu_1 -\mu_3)(\mu_4 - \mu_2)},
\end{equation}
if we fix the order of the principal curvatures of $M$ to be,
\begin{equation}
\label{eq:pc-order}
\mu_1 < \mu_2 < \mu_3 < \mu_4.
\end{equation}
For an isoparametric
hypersurface with four principal curvatures ordered as in equation (\ref{eq:pc-order}),
M\"{u}nzner's results \cite{Mu}--\cite{Mu2} imply that
the Lie curvature
$\psi = 1/2$, and
the multiplicities satisfy $m_1 = m_3$, $m_2 = m_4$.
Furthermore, if $M \subset S^n$ is a compact, connected proper Dupin hypersurface 
with $g=4$ principal curvatures, then the multiplicities of the principal curvatures must be the same as those of an isoparametric hypersurface by the work of Stolz \cite{Stolz}, so they satisfy $m_1 = m_3$, $m_2 = m_4$.

Cecil-Chi-Jensen \cite{CCJ2} proved the following 
local classification of irreducible proper Dupin hypersurfaces with four principal curvatures and constant Lie curvature $\psi = 1/2$. In the case where all the multiplicities equal one, this theorem was first proven
by Cecil and Jensen \cite{CJ3}.

\begin{theorem}
\label{CCJ} 
(Cecil-Chi-Jensen, 2007)
Let $M \subset S^n$ be a connected irreducible proper Dupin hypersurface with four principal curvatures 
\begin{displaymath}
\mu_1 < \mu_2 < \mu_3 < \mu_4,
\end{displaymath}
having multiplicities,
\begin{equation}
\label{eq:restricted}
m_1 = m_3 \geq 1, \quad m_2 = m_4 =1,
\end{equation}
and constant Lie curvature $\psi = 1/2$. Then $M$ is Lie equivalent to an
isoparametric hypersurface in $S^n$.
\end{theorem}

By Theorem \ref{thm:CCJ-2007} above, we know that
compactness implies irreducibility for proper Dupin hypersurfaces with more than two principal curvatures.
Furthermore, Miyaoka \cite{Mi2} proved that if $\psi$ is constant on a compact proper Dupin hypersurface $M \subset S^n$ with $g=4$, then $\psi = 1/2$ on $M$, when
the principal curvatures are ordered as in equation (\ref{eq:pc-order}).  
As a consequence, we get the following corollary of
Theorem \ref{CCJ}, which proves that Conjecture \ref{revised-conjecture} is true in a special case.

\begin{corollary}
\label{cor:CCJ-compact}
(Cecil-Chi-Jensen, 2007)
Let $M \subset S^n$ be a compact, connected proper Dupin hypersurface with 
four principal curvatures 
\begin{displaymath}
\label{eq:pc-order-2}
\mu_1 < \mu_2 < \mu_3 < \mu_4,
\end{displaymath}
having multiplicities
\begin{displaymath}
m_1 = m_3 \geq 1, \quad m_2 = m_4 =1,
\end{displaymath}
and constant Lie curvature $\psi$.  Then $M$ is Lie equivalent to an isoparametric hypersurface in $S^n$.
\end{corollary}

A remaining open question is to resolve what happens in
the general case where $m_2 = m_4$ is also allowed to be greater than one, i.e.,
\begin{equation}
\label{eq:general}
m_1 = m_3 \geq 1,\quad  \ m_2 = m_4 \geq 1,
\end{equation}
and constant Lie curvature $\psi$?\\

\noindent
Regarding this question, we note that
the local proof of Theorem \ref{CCJ} uses the method of moving frames, and it involves a large system
of equations that contains certain sums if some $m_i$ is greater than one, but no 
corresponding sums if all $m_i$ equal one.  These sums make the calculations significantly more difficult, and
so far this method has not led to a proof in the general case (\ref{eq:general}).

Key elements in the proof of Theorem \ref{CCJ}
are Pinkall's \cite{P2} Lie geometric criteria for reducibility (Theorem 
\ref{thm:pinkall-reducibility}), and the Lie geometric criteria for Lie equivalence to
an isoparametric hypersurface (Theorem \ref{thm:lie-eq-isop-hyp}), due to Cecil \cite{CecKod} (see also
\cite[p. 77]{Cec9}).\\

\noindent
Case $g = 6$:
Grove and Halperin (1987) proved that if $M \subset S^n$ is a compact proper Dupin hypersurface
with $g=6$ principal curvatures, then all the principal curvatures must have the same multiplicity $m$, and
$m = 1$ or 2, as is the case for an isoparametric hypersurface, as shown by Abresch \cite{Ab}.
They also proved other topological 
results about compact proper Dupin hypersurfaces that support Conjecture \ref{revised-conjecture} in the case $g=6$.

As mentioned in Remark \ref{rem:miyaoka-examples} above,
Miyaoka \cite{Mi3} showed that if some additional assumptions are made regarding the intersections of the leaves of the various principal foliations, then Conjecture \ref{revised-conjecture} is true in the case $g=6$.  
However, it has not been proven that Miyaoka's additional assumptions are satisfied in general,
and so Conjecture \ref{revised-conjecture} remains as an open problem in the case $g=6$.\\

\section{Local Results on Dupin Submanifolds}
\label{sec:local-dupin}

Most local classifications of proper Dupin submanifolds have 
been obtained in the context of Lie sphere geometry.  We will state
these results for hypersurfaces of $S^n$, since we can
always arrange that the point sphere map of a Legendre submanifold
is locally an immersion by taking a parallel submanifold if
necessary.  

The known results depend on the number $g$ of
distinct principal curvatures, and they are progressively harder
to prove as $g$ increases.  In fact, results have only been
obtained up to the case $g=4$, and more remains to be done
in that case.  

Of course, 
a connected proper Dupin hypersurface in $S^n$ 
with one distinct principal curvature must be an open
subset of a hypersphere.  
In the case $g=2$, Pinkall \cite{P2} obtained a complete local classification which we now describe.

A proper Dupin hypersurface in $S^n$ (or ${\bf R}^n$)
with two distinct principal curvatures of respective 
multiplicities $p$ and $q$ is called a
{\em cyclide of Dupin of characteristic} $(p,q)$.
The compact cyclides embedded in ${\bf R}^3$ can all be obtained
through stereographic projection from a standard product of two
circles in the unit sphere $S^3 \subset {\bf R}^4$.  This
construction obviously can be generalized to higher dimensions.

Cecil and Ryan \cite{CRMA} showed that a connected, compact
cyclide $M^{n-1}$ of characteristic $(p,q)$ embedded in
$S^n$ must be M\"{o}bius equivalent to a standard product of spheres,
\begin{equation}
\label{eq:standard-product-spheres}
S^p (r) \times S^q  (s) \subset S^n (1) \subset {\bf R}^{p+1}
\times {\bf R}^{q+1} = {\bf R}^{n+1}, \quad r^2 + s^2 = 1,
\end{equation}
where $n = p+q+1$.  Varying the value of $r$ in \eqref{eq:standard-product-spheres} produces
a family of parallel hypersurfaces.  These are Lie equivalent
by parallel transformation, but if $p \neq q$, they are not M\"{o}bius equivalent
for different values of $r$.

\begin{remark} 
\label{rem:case-p=q}
{\rm As noted by the referee, in the case $p = q$, the statement that different values of $r$ give M\"{o}bius inequivalent products requires 
qualification.  Set
\begin{equation}
\label{eq:standard-product-spheres-p=q}
M_{r,s} = S^p (r) \times S^p  (s) \subset S^{2p+1} (1), \quad r,s > 0, \quad r^2 + s^2 = 1,
\end{equation}
as in equation \eqref{eq:standard-product-spheres}. 
The orthogonal transformation $(x,y) \mapsto (y,x)$ maps $M_{r,s}$ to $M_{s,r}$. Thus, even when $r \neq s$, these two 
products are congruent and therefore M\"{o}bius equivalent.}
\end{remark}

The proof of Cecil and Ryan \cite{CRMA}
uses the assumption of compactness in an essential way, whereas the classification of Dupin surfaces
in ${\bf R}^3$ obtained in the nineteenth century does
not need such an assumption (see, for example, \cite[pp. 151-166]{CR7}).  
Using Lie sphere geometry,
Pinkall \cite{P2} obtained the following classification in arbitrary 
dimensions which also does
not need the assumption of compactness (see also \cite[pp. 148--159]{Cec9}).
\begin{theorem}
\label{thm:pinkall-cyclides}
{\em (a)} Every connected cyclide of Dupin is contained in a
unique compact, connected cyclide.\\
{\em (b)} Any two cyclides of the same characteristic are locally
Lie equivalent.
\end{theorem}
A consequence of part (a) of the theorem is that
any connected
piece of a cyclide of Dupin of characteristic $(p,q)$ immersed
in $S^n$ determines a unique compact Legendre submanifold 
with domain $S^p \times S^q$.
Pinkall's result can be used to obtain the following M\"{o}bius
geometric characterization of the cyclides in ${\bf R}^n$ 
(see \cite{CecL} or \cite[p. 151]{Cec9}).

\begin{theorem}
\label{thm:cyclides-moebius}
{\em (a)} Every connected cyclide of Dupin $M^{n-1}$ of characteristic
$(p,q)$ embedded in ${\bf R}^n$ is M\"{o}bius equivalent to an open
subset of a hypersurface of revolution obtained by revolving a
q-sphere $S^q \subset {\bf R}^{q+1} \subset {\bf R}^n$ about an axis
of revolution ${\bf R}^q \subset {\bf R}^{q+1}$, or a p-sphere $S^p
\subset {\bf R}^{p+1} \subset {\bf R}^n$ about an axis
${\bf R}^p \subset {\bf R}^{p+1}$.\\
{\em (b)} Two such hypersurfaces are M\"{o}bius equivalent if and only if
they have the same value of $\rho = |r|/a$, where $r$ is the signed
radius of the profile sphere $S^q$ and $a>0$ is the distance from
the center
of $S^q$ to the axis of revolution.
\end{theorem}
Note that the profile sphere is allowed to intersect the axis
of revolution, thereby resulting in singularities.  However,
in the context of Lie sphere geometry, the corresponding Legendre
map is an immersion.

\begin{remark} 
\label{rem:cyclides-ivey}
{\rm The classical cyclides of Dupin in ${\bf R}^3$ are the only surfaces
for which all lines of curvature are circles or straight lines.
Using exterior differential systems,
Ivey \cite{Ivey} showed that any surface in ${\bf R}^3$ containing
two orthogonal families of circles  is a cyclide of Dupin.}
\end{remark}

Finally, we note that the cyclides of Dupin have been
used in computer aided geometric design of surfaces.  See,
for example, the papers of Degen \cite{Degen}, Pratt 
\cite{Pr1}--\cite{Pr2}, and Srinivas and Dutta \cite{SD1}--\cite{SD4}.

The case $g=3$ has proven to be much more difficult than the
case of two principal curvatures.  In his
dissertation, Pinkall \cite{P1} (see also Pinkall
\cite{P3}, Cecil-Chern \cite{CC2} or the book \cite[pp. 168--190]{Cec9}) obtained
a complete local classification up to Lie
sphere transformation for proper Dupin hypersurfaces 
with three principal curvatures in ${\bf R}^4$. 
Pinkall found that any two irreducible
proper Dupin hypersurfaces with $g=3$ in ${\bf R}^4$ are locally
Lie equivalent, each being Lie equivalent to an open subset
of an isoparametric hypersurface in $S^4$.  However, he found
a 1-parameter family of Lie equivalence classes among the
reducible proper Dupin hypersurfaces with $g=3$ in ${\bf R}^4$.

In two papers \cite{CJ2}--\cite{CJ3}, Cecil and Jensen
used a notion called local irreducibility in the formulation of certain classification
results.  Specifically, a proper Dupin submanifold $\lambda: M^{n-1} \rightarrow \Lambda^{2n-1}$
is said to be {\em locally irreducible},
if there does not exist any open subset $U \subset M^{n-1}$ such that
the restriction of $\lambda$ to $U$ is reducible.

Theoretically, this is a stronger condition than irreducibility of $\lambda$ itself.  However,
using the analyticity of proper Dupin submanifolds, Cecil, Chi and Jensen \cite{CCJ2}
 proved the following
result which shows that the concepts of local irreducibility and irreducibility are equivalent.
(See also \cite[pp. 145--146]{Cec9} for a complete proof of this result.)

\begin{theorem}
\label{prop:4.2.10a} 
Let $\lambda: M^{n-1} \rightarrow \Lambda^{2n-1}$ be a connected, proper Dupin submanifold.  If the restriction of
$\lambda$ to an open subset $U \subset M^{n-1}$ is reducible, then $\lambda$ is reducible.  Thus, a connected
proper Dupin submanifold is locally irreducible if and only if it is irreducible.
\end{theorem}

In higher dimensions, the focus has been on classifying the irreducible 
proper Dupin hypersurfaces and little has
been done in attempting to classify the reducible ones up
to Lie equivalence, although this appears to be a problem
where some progress could be made.

The first result in higher dimensions is due to Niebergall
\cite{N1} who showed that every proper Dupin hypersurface
in ${\bf R}^5$ with three principal curvatures is reducible.
Cecil and Jensen \cite{CJ2} then generalized the
results of Pinkall and Niebergall as follows. 
\begin{theorem}
\label{thm:CJ-3-p-c}
Let $f:M \rightarrow S^n$ be a connected proper Dupin hypersurface with
three distinct principal curvatures of multiplicities
$m_1, m_2,$ and $m_3$, respectively.  If the hypersurface $M$ is
irreducible, then $m_1 = m_2 = m_3$, and $M$ is Lie 
equivalent to an isoparametric hypersurface in a sphere.
\end{theorem}

We now give a brief outline of the proof of this theorem.
We work in the context of Lie sphere geometry and consider
a proper Dupin submanifold $\lambda :M^{n-1} \rightarrow \Lambda ^{2n-1}$
with three curvature spheres.  As in Section \ref{sec:submanifolds-lie-geometry}, we can parametrize
the Dupin submanifold as
$\lambda = [k_1 , k_2 ]$, where $[k_1]$ and $[k_2]$ are two
curvature sphere maps.  We can also arrange that the third
curvature sphere map has the form $[k_3] = [k_1 + k_2]$.  

We first compute the derivatives of the $[k_i]$ using the method
of moving frames.  In the case where $M$ has dimension three, Pinkall \cite{P3}
found one function $c$ with the property that all of the terms
arising in the exterior differentiation 
of the frame fields could eventually be expressed
in terms of $c$ and its derivatives.  If $c$ is identically zero,
then $\lambda$ is reducible.  If $c$ is never zero on $M$, then
one can arrange that $c=1$ with an appropriate choice of frame,
and all such hypersurfaces are locally Lie equivalent to Cartan's
isoparametric hypersurface in $S^4$.

In the general case where the three curvature spheres have respective
multiplicities $m_1, m_2$ and $m_3$, there are
$m_1 \cdot m_2 \cdot m_3$ functions 
$F^{\alpha}_{ap},$ where
\[1\leq a \leq m_1, \hspace{.25in} m_1 + 1 \leq p \leq m_1 + m_2,
\hspace{.25in} m_1 + m_2 + 1 \leq \alpha \leq n-1,\]
which are defined in a similar
way to the one function $c$ of Pinkall. These can be arranged
in vector form, 
$v_{p\alpha} = (F^{\alpha}_{ap}),  1 \leq a \leq m_1,$ with
$v_{a\alpha}$ and  $v_{ap}$ defined in a similar way.  One first
shows that if a column or row of any of the arrays 
$[v_{p\alpha}]$, $[v_{a\alpha}]$, $[v_{ap}]$ is identically
zero on an open subset $V \subset M$, then the restriction of
$\lambda$ to $V$ is reducible.
This result is then applied to show
that unless all of the multiplicities are equal, 
$\lambda$ must contain a reducible open subset,  and so it is reducible.

In the case where all the multiplicities equal $m$,
one next shows that all of the vectors in all of the
arrays have the same length $\rho$.  This one function $\rho$
actually plays the same role that $c$ did in the case $m=1$.
If $\lambda$ is irreducible, then $\rho$ is non-zero on $M$, and
it can be made locally constant with an appropriate choice of frame.
The proof is completed by showing that there exist three points 
$P_1, P_2, P_3$ on a certain timelike line in ${\bf RP}^{n+2}$ such that
$\langle k_i, P_i \rangle = 0,$ for $1\leq i \leq 3$, where $[k_i]$ are the
curvature sphere maps of $\lambda$.  By Theorem \ref{thm:lie-eq-isop-hyp}, this implies that
$\lambda$ is Lie equivalent to an isoparametric hypersurface.

The next case $g=4$ is still more complicated,
but many aspects of the approach outlined above apply.  As in Section \ref{sec:compact-proper-dupin},
one can order the principal
curvatures as in equation \eqref{eq:pc-order}
and determine a unique Lie curvature $\psi$ by equation \eqref{eq:lie-curv} which satisfies
$0< \psi < 1$.  As in the
$g=3$ case, one can reparametrize the Dupin submanifold as
$\lambda = [k_1, k_2]$, where $[k_1]$ and $[k_2]$ are two of the
curvature sphere maps,  and then arrange that a third curvature
sphere map satisfies $[k_3] = [k_1+k_2]$.  Then the fourth curvature
sphere $[k_4]$ is determined by the Lie curvature $\psi$.  

For $g=4$ or 6, it is not true that every irreducible Dupin 
hypersurface is Lie equivalent to an isoparametric hypersurface.
The examples of Pinkall and Thorbergsson \cite{PT} and Miyaoka and Ozawa
\cite{MO} discussed in Section \ref{sec:compact-proper-dupin} are irreducible and are not
Lie equivalent to an isoparametric hypersurface, because they have non-constant Lie curvatures.

As noted in Section \ref{sec:compact-proper-dupin}, 
this leads to Conjecture \ref{revised-conjecture-local} above which states that
every irreducible connected proper Dupin hypersurface in $S^n$ with four or six principal curvatures
and constant Lie curvatures is Lie equivalent to an isoparametric hypersurface.
See the discussion in Section \ref{sec:compact-proper-dupin} following 
Conjecture \ref{revised-conjecture-local} for progress that has been made on that conjecture,
in particular, Theorem \ref{CCJ}  and Corollary \ref{cor:CCJ-compact}.

A more general problem is to attempt to identify key local
invariants of Dupin submanifolds within the context of moving
Lie frames, as was done in the paper of
Niebergall \cite{N2}, and the papers of Cecil and Jensen \cite{CJ2}--\cite{CJ3},
and Cecil, Chi, and Jensen \cite{CCJ2}.

Another problem is to attempt to obtain a complete local
classification of reducible Dupin hypersurfaces of arbitrary dimensions
with $g=3$ principal curvatures up to Lie equivalence. 
As mentioned above, Pinkall obtained such a classification in the 
case of $M^3 \subset {\bf R}^4$.  In that case, while there is only
one class of irreducible Dupin hypersurfaces, the reducible ones 
determine a 1-parameter family of Lie equivalence classes.  It may be
possible to obtain a similar classification in the reducible case in some higher dimensions
by using the framework established
to prove Theorem \ref{thm:CJ-3-p-c}.

The approach of Lie sphere geometry can also be used
to obtain results in M\"{o}bius (conformal) geometry.  As noted in
Theorem \ref{thm:cyclides-moebius},
one can derive a local M\"{o}bius classification in the case $g=2$
from Pinkall's Lie-geometric classification.  In \cite{PT}
Pinkall and Thorbergsson introduced a M\"{o}bius invariant called
the {\em M\"{o}bius curvature} which can distinguish among the Lie
equivalent parallel hypersurfaces in a family of isoparametric
hypersurfaces.  

C.-P. Wang used the method
of moving frames to determine a complete
set of M\"{o}bius invariants for surfaces in ${\bf R}^3$ without
umbilic points \cite{Wc}, and for hypersurfaces in ${\bf R}^4$
with three distinct principal curvatures at each point \cite{Wc1}.
He then applied this result to derive a local
classification of Dupin hypersurfaces in ${\bf R}^4$ with three
principal curvatures
up to M\"{o}bius transformation.  

In two related papers, Ferapontov \cite{Fera}--\cite{Fera1} 
explored the
relationship between Dupin and 
isoparametric hypersurfaces and Hamiltonian
systems of hydrodynamic type.  Ferapontov posed several research
problems in that context. 

Motivated by  Bryant's \cite{Bryant} research on austere subspaces, and by Cartan's isoparametric hypersurfaces with three distinct principal curvatures (see Section \ref{sec:isoparametric-submanifolds}),
Ge and Zhou \cite{Ge-Zhou} constructed three families of austere submanifolds with flat normal bundle in unit spheres.  From those examples, they found three irreducible proper Dupin hypersurfaces with
five distinct principal curvatures of different multiplicities.  As an application, they obtained an upper bound estimate for the dimension of austere subspaces.\\

\section{Classifications of Taut Submanifolds}
\label{sec:classifications-taut-submanifolds}

In this section, we survey the known classification results
on taut submanifolds.  To a reasonable extent, we have attempted to 
make this section self-contained, although some references
to the previous sections are inevitable.  In addition to the presentation in the
original version of this paper, 
we will also closely follow the presentation in the book
\cite[pp. 327--342]{CR8} regarding more recent results. 
In general, results have been obtained for manifolds with relatively 
simple homology, but general classifications of taut submanifolds are fairly rare.

In the paper which introduced the {\em STPP},
Banchoff \cite{Ban1} showed that a taut embedding of
$S^1$ into ${\bf R}^n$ must be a metric circle in a plane.
In the same paper,
Banchoff  also obtained a complete classification of
compact taut (2-dimensional) surfaces in Euclidean spaces.
We now give a brief outline of his results.  

As noted in Section \ref{sec:taut-submanifolds}, Banchoff observed that
because tautness is invariant under stereographic projection,
there exists a substantial taut
non-spherical embedding of a compact manifold $M$
into ${\bf R}^n$ if and only if there exists a substantial
taut spherical embedding of $M$ into $S^n \subset {\bf R}^{n+1}$. 
As a consequence, one can invoke
Kuiper's result on the bound on the codimension 
to obtain Theorem \ref{thm:kuiper-bound-on-codimension}.  In particular,
if $f:M^2 \rightarrow {\bf R}^n$ is a substantial taut embedding, then
$n \leq 5$, and if $n=5$, then $f$ is an embedding of ${\bf RP}^2$ as 
a spherical Veronese surface in $S^4 \subset {\bf R}^5$. 
(Recall that a spherical Veronese surface in $S^4 \subset {\bf R}^5$ is the well-known 
standard embedding of the real projective plane
${\bf RP}^2$ into a metric sphere $S^4 \subset {\bf R}^5$, see Remark \ref{rem:std-embeddingd-projective-spaces}.)

Next, if 
$f:M^2 \rightarrow {\bf R}^4$ is a substantial taut non-spherical embedding, then
$f(M^2)$ must be the image under stereographic projection of
a spherical Veronese surface in ${\bf R}^5$.  Thus, the problem is
reduced to finding all compact taut surfaces in $S^3$.  Again by
stereographic projection, this is equivalent to finding all
compact taut surfaces in ${\bf R}^3$.

A key step in Banchoff's classification of compact taut surfaces
in ${\bf R}^3$ is showing that if $f:M^2 \rightarrow {\bf R}^3$
is a taut embedding of a compact surface $M^2$, then 
$f(M^2)$ lies in between the two spheres $S_1$ and $S_2$
tangent to $f(M^2)$ at $f(x)$ and centered at the focal
points $p_i = f(x) + (1/\mu_i)\xi(x)$, for $i=1,2$, respectively,
where $\xi$ is a field of unit normals to $f(M^2)$, and
the $\mu_i$ are the principal curvatures of $f$.  Thus,
if $f(M^2)$ has one umbilic point, then it must be a metric
sphere, because it lies between two identical spheres
$S_1$ and $S_2$ at the umbilic point.  This implies that
$f(M^2)$ must be either a metric sphere or smooth torus, because
any embedding of a surface of higher genus would necessarily have an
umbilic point.  Suppose now that $f(M^2)$ is a taut torus with no
umbilic points.  Then Banchoff shows that the principal curvatures must
be constant along their corresponding lines of curvature, i.e.,
$f(M^2)$ is proper
Dupin, and so it is a cyclide of Dupin in ${\bf R}^3$. 
 
Cecil \cite{CecJDG} then generalized Banchoff's argument to
the non-compact case to again show that a taut
$f(M^2) \subset {\bf R}^3$ 
must be proper Dupin.  This implies that $f(M^2)$ must be a plane, circular
cylinder or parabolic ring cyclide, which is obtained 
by inverting a torus of revolution in a sphere centered at a point
on the torus.  Its name comes from the fact that its focal set consists
of a pair of parabolas (see \cite[pp. 151--166]{CR7} for more detail).
These results are combined in the
following theorem, which is a complete classification of
taut surfaces in  ${\bf R}^n$.

\begin{theorem}
\label{thm:taut-surfaces}
Let $f:M^2 \rightarrow {\bf R}^n$ be a substantial taut embedding of a
surface $M^2$. \\ 
{\rm (a)} If $M^2$ is compact, then $f(M^2)$ is a metric sphere or
a cyclide of Dupin in ${\bf R}^3$, a spherical Veronese surface embedded in
$S^4 \subset {\bf R}^5$,
or a surface
in ${\bf R}^4$ related to one of these by 
stereographic projection.\\
{\rm (b)} If $M^2$ is non-compact, then it is a plane, circular cylinder
or parabolic ring cyclide in ${\bf R}^3$, or it is the image in
${\bf R}^4$ of a punctured spherical Veronese surface under
stereographic projection from $S^4$, where the pole of the stereographic projection lies on
the spherical Veronese surface in $S^4$.
\end{theorem}

The first results after Banchoff's paper dealt with taut embeddings
of submanifolds with relatively simple topology.  Nomizu and
Rodriguez \cite{NR} proved the following.
\begin{theorem}
\label{thm:nomizu-rodriguez}
Let $M^k, k \geq 2,$ be a complete Riemannian manifold
isometrically immersed in  ${\bf R}^n$.
Every Morse function of the form $L_p$ has index 0 or $k$ at each of its critical points if and only if $M^k$
is embedded as a $k$-plane or a
metric $k$-sphere $S^k \subset {\bf R}^{k+1} \subset {\bf R}^n$.
\end{theorem}
\begin{proof}
The proof is accomplished by showing that $M^k$ is a totally
umbilic submanifold.  This is a consequence of the
following elementary but important argument.  Let $f:M^k \rightarrow
{\bf R}^n$ be the isometric immersion.
Let $\xi$
be any unit normal to $f(M^k)$ at any point $f(x)$.  
We want to show that 
the shape operator $A_{\xi}$ is a scalar multiple of the identity.
If $A_{\xi}=0$, then we are done.  If not, then by replacing
$\xi$ by $-\xi$, if necessary, we can assume that $A_{\xi}$ has
a positive eigenvalue.  Let $\lambda$ be the largest eigenvalue
of $A_{\xi}$, and let $t$ satisfy $1/\lambda <t< 1/\mu$, where
$\mu$ is the next largest positive eigenvalue of $A_{\xi}$
(just consider $t > 1/\lambda$ if there are no other positive
eigenvalues).  If $q=f(x) + t\xi$, then $L_q$ has a non-degenerate
critical point at $x$ with index equal to the multiplicity $m$
of $\lambda$.  It may be that $L_q$ is not a Morse function. If so, 
then there exists a Morse function
$L_p$, with $p$ near $q$, such that $L_p$ has a critical 
point near $x$ of the same index $m$. Now since $m>0$,
the hypotheses imply that $m=k$, and thus $A_{\xi} = \lambda I$,
as desired.
\end{proof}

An immediate consequence is the following result obtained
independently by Carter and West \cite{CW1}. The result
is true in the case $k=1$ by the work of Banchoff mentioned 
earlier.  See Cecil \cite{CecToh} for a similar characterization
of metric spheres in hyperbolic space.
\begin{theorem}
\label{thm:carter-west-taut-spheres}
Let $f:S^k \rightarrow {\bf R}^n$ be a taut embedding.
Then $f(S^k)$ is a metric sphere in a $(k+1)$-dimensional
Euclidean subspace ${\bf R}^{k+1} \subset {\bf R}^n$.
\end{theorem}

\begin{remark}{Tight and taut immersions into hyperbolic space}
\label{rem:tight-taut-hyp-space}

\noindent
{\rm In hyperbolic space $H^m$ there are three types of totally umbilic hypersurfaces: spheres, horospheres
and equidistant hypersurfaces (those at a fixed oriented distance from a totally geodesic
hyperplane, including hyperplanes themselves). 
These have constant sectional curvature which is positive, zero, or negative, for spheres, horospheres and equidistant hypersurfaces, respectively.  Thus, there are three natural types of distance
functions $L_p$, $L_h$ and $L_\pi$, which measure the distance
from a given point $p$, horosphere $h$, or hyperplane $\pi$, respectively. Just as in Euclidean space 
(see Theorem \ref{thm:nomizu-rodriguez}),
the totally umbilic hypersurfaces of $H^m$ can be characterized in terms of the critical point behavior of these distance functions as follows (see Cecil-Ryan \cite{CRNag}).

\begin{theorem}
\label{thm:C-R-10.1}
Let $M^n, n\geq 2$, be a connected, complete Riemannian manifold isometrically immersed in $H^m$. Every Morse function of the form $L_p$ or $L_\pi$ has index 0 or $n$ at each of its critical points if and only if $M$
is embedded as a sphere, horosphere, or equidistant hypersurface in a totally geodesic $H^{n+1} \subset H^m$.
\end{theorem}
An immersion $f:M \rightarrow H^m$ is called {\em taut}, 
{\em horo-tight}, or {\em tight}, respectively, 
if every non-degenerate function $L_p$, $L_h$, or $L_\pi$, has the minimum number of critical points
required by the Morse inequalities.  See Cecil and Ryan \cite{CRNag}, \cite{CRLon}, \cite[pp. 233--236]{CR7}, and Izumiya et al. \cite{Izumiya2005}--\cite{Izumiya2003},
for more on these conditions. }
\end{remark}

Approaching the problem from a different point of
view, Hebda \cite{Heb1} asked which ambient spaces admit taut embeddings
of hyperspheres.
He found that a complete simply connected $n$-dimensional manifold
which admits a taut embedding of $S^{n-1}$ is either homeomorphic
to $S^n$, diffeomorphic to ${\bf R}^n$ or diffeomorphic to
$S^{n-1} \times {\bf R}$.

Next Carter and West \cite{CW1} obtained the following characterization
of spherical cylinders.
\begin{theorem}
\label{thm:carter-west-non-compact}
Let $f:M^{n-1} \rightarrow {\bf R}^n$ be a taut embedding of a
non-compact manifold such that $H_k(M^{n-1};{\bf Z}_2) = 
{\bf Z}_2$ for some $k,\  0 < k < n-1$, and $H_i(M^{n-1};{\bf Z}_2) = 0$,
for $i \neq 0,k$.  Then $M^{n-1}$ is diffeomorphic to
$S^k \times {\bf R}^{n-k-1}$, and $f$ is a standard product embedding.
\end{theorem}
Thorbergsson \cite{Th3}--\cite{Th4}
then obtained the following characterization of
highly connected taut submanifolds of arbitrary codimension.
(See L. Jorge and F. Mercuri \cite{JM} for a related result involving minimal submanifolds.)
\begin{theorem}
\label{thm:thorbergsson-highly-connected-taut}
Let $M^{2k}$ be a compact $(k-1)$-connected but not $k$-connected
taut submanifold of ${\bf R}^n$ which does not lie in any
totally umbilic hypersurface of ${\bf R}^n$, then either:\\
{\rm (a)} $n = 2k+1$ and $M^{2k}$ is a cyclide of Dupin diffeomorphic
to $S^k \times S^k$ or
{\rm (b)} $n = 3k+1$ and $M^{2k}$ is a standard embedding of the 
projective plane ${\bf FP}^2$, 
where ${\bf F}$ is the division algebra
{\bf R}, {\bf C}, {\bf H} (quaternions),
{\bf O} (Cayley numbers) for $k=1,2,4,8,$ respectively.
\end{theorem}
In a related paper, Hebda \cite{Heb2} constructed tight smooth
embeddings of arbitrarily many copies of $S^k \times S^k$ into
${\bf R}^{2k+1}$.  No taut embeddings of these manifolds
exist by Theorem \ref{thm:thorbergsson-highly-connected-taut}.
 
The next case in terms of the homology is when $M$ has
the same homology as $S^k \times S^m$, where $k \neq m$.
In that case, Cecil and Ryan \cite{CRMA} (see also \cite[pp. 202--206]{CR7}, \cite[pp. 332--338]{CR8})
obtained the following.
\begin{theorem}
\label{thm:C-R-taut-product-spheres}
A taut hypersurface $M \subset {\bf R}^n$ with the same
${\bf Z}_2$-homology as $S^k \times S^{n-k-1}$ is a cyclide
of Dupin.
\end{theorem}
To prove this,
one first uses the Index Theorem for $L_p$ functions to prove that
at each point of $M$, the number of distinct principal curvatures
must be either 2 or 3.  The most difficult part of the proof is to then
show that the number of distinct principal curvatures
must be constant on $M$.  Then since taut implies Dupin,
$M$ is a proper Dupin
hypersurface with $g =2$  or 3 principal curvatures.
Then it is fairly easy given the homology of $M$
to show that $g=2$, and so $M$ must be a cyclide of
Dupin.

Ozawa \cite{Oz} generalized this result by showing that
if an embedding of $S^k \times S^m, k < m$, into $S^n$
is taut and substantial, then the codimension of the embedding
is either 1 or $m-k+1$. He also showed that the $r$-times
connected sum of $S^k \times S^m, k < m$, cannot be
tautly embedded into any Euclidean space if $r>1$.

There is a related result which takes into account the
intrinsic geometry of $M$.  Recall that a Riemannian manifold
$(M,g)$ is said to be {\em conformally flat} if every point
has a neighborhood conformal to an open subset in Euclidean
space.  Schouten \cite{Sch} showed that a hypersurface
$M^n, n \geq 4$, 
immersed in ${\bf R}^{n+1}$ is conformally flat in the
induced metric if and only if at least $n-1$ of the principal
curvatures coincide at each point. (This characterization
fails when $n=3$, see Lancaster \cite{Lan}.)  Using Schouten's
result, Theorem \ref{thm:C-R-taut-product-spheres} and some basic results on tautness,
Cecil and Ryan \cite{CRCan} proved the following.
\begin{theorem}
\label{thm:C-R-conformally-flat}
Let $M^n, n \geq 4,$ be a connected manifold tautly embedded in ${\bf R}^{n+1}$.
Then $M^n$ is conformally flat in the induced metric if and
only if it is one of the following:\\
{\rm (a)} a hyperplane or metric sphere;\\
{\rm (b)} a cylinder over a circle or over an $(n-1)$-sphere;\\
{\rm (c)} a ring cyclide diffeomorphic to $S^1 \times S^{n-1}$;\\
{\rm (d)} a parabolic ring cyclide diffeomorphic to $(S^1 \times S^{n-1}) - \{p\}$.
\end{theorem}

In certain specific 
classification results involving  taut submanifolds in $S^n$ or ${\bf R}^n$, an important step is proving that 
the complement of the focal set in the ambient space is connected (see, for example, Carter-West \cite{CW1}, Cecil-Ryan \cite{CRMA}, \cite[pp. 336--338]{CR8}).  In 2008, Cecil, Chi and Jensen \cite[p. 237]{CCJ3} showed that this is always the case by proving the following theorem (see also \cite[pp. 339--340]{CR8} for a proof).

\begin{theorem}
\label{thm:CR-8.11}
Let $M \subset S^n$ be a connected taut submanifold of $S^n$. Then the complement of the focal set of $M$ in $S^n$ is connected.
\end{theorem}

\subsubsection*{Classifications of taut 3-manifolds}

Concerning taut embeddings of 3-manifolds, Pinkall and
Thorbergsson \cite{PT1} have proven the following result.
\begin{theorem}
\label{thm-Pinkall-Thorb-3d}
A compact taut 3-dimensional submanifold in Euclidean space
is diffeomorphic to one of the following seven manifolds:\\
$S^3$, ${\bf RP}^3$, the quaternion space $S^3/\{\pm 1, \pm i,
\pm j, \pm k\}$, the 3-torus $T^3$,
$S^1 \times S^2$, $S^1 \times {\bf RP}^2$,
$S^1 \times_h S^2$, where $h$ denotes an orientation reversing
diffeomorphism of $S^2$.
Furthermore, all of these manifolds admit taut embeddings.
\end{theorem}
Pinkall and Thorbergsson \cite{PT1} gave more detail about these embeddings,
and we summarize their results and examples below.
Since tautness is invariant under stereographic projection, they
classified {\em spherically substantial} taut embeddings, i.e.,
those which do not lie in any hypersphere.  In the description below,
the codimension means the spherically substantial
codimension.

A taut embedding of $S^3$ is a metric hypersphere, as noted in Theorem \ref{thm:carter-west-taut-spheres}. Real 
projective space  ${\bf RP}^3$
can be tautly embedded with codimension 2 as the Stiefel
manifold $V_{2,3} \subset S^5 \subset {\bf R}^6$ of orthonormal 2-frames in ${\bf R}^3$,
and with
codimension 5 as $SO(3)$ in the unit sphere in the space
of $3 \times 3$ matrices.  Pinkall and Thorbergsson did not determine whether the codimensions
3 and 4 are possible.  

The quaternion space is embedded as
Cartan's isoparametric hypersurface in $S^4$ (see \cite[pp. 151--155]{CR8}), where it is unique
up to Lie equivalence, and no other codimensions are possible.
The 3-torus can be tautly embedded with codimension one as a
tube in ${\bf R}^4$ around a torus of revolution 
$T^2 \subset {\bf R}^3 \subset {\bf R}^4$, and with codimension 2 as a product
$T^2 \times S^1 \subset {\bf R}^5$. 

The space $S^1 \times S^2$ can be tautly embedded
with codimension 1
as a cyclide of Dupin (see Theorem \ref{thm:C-R-taut-product-spheres}), and no other codimension
is possible.  The manifold  $S^1 \times {\bf RP}^2$ can be
tautly embedded with substantial codimension 3 as the product of a 
metric circle and a spherical Veronese surface. It can be tautly
embedded with codimension 2 as a rotational
submanifold with profile submanifold  ${\bf RP}^2$, 
and the only codimensions
possible are 2 and 3.  

Finally, $S^1 \times_h S^2$ can be tautly embedded with codimension 2 as the
``complexified unit sphere"
\begin{equation}
\label{eq:complexified-unit-sphere}
\{e^{i\theta} x \mid \theta \in {\bf R}, x \in S^2 \subset {\bf R}^3\}
\subset S^5 \subset {\bf C}^3.
\end{equation}
This is one of the focal submanifolds of Cartan's \cite{Car4} homogeneous family
of isoparametric hypersurfaces with four principal curvatures
in $S^5$ (see also \cite[pp. 155--159]{CR8}).
No other codimensions are possible for
a taut embedding of $S^1 \times_h S^2$.

\subsubsection*{Other results on taut submanifolds}

\begin{remark}{Taut embeddings of 4-manifolds}
\label{rem:taut-4-man}

\noindent
{\rm In a nice survey article on taut submanifolds, Gorodski \cite{Gorodski-06} 
obtained some partial results on taut embeddings of 4-manifolds into spheres.  Let $M$ be a compact, connected smooth 4-dimensional taut
submanifold of a sphere $S^n$ for some $n$.  Gorodski showed that if $M$ has vanishing first Betti number, then $M$ is diffeomorphic to $S^4$, $S^2 \times S^2$ or ${\bf CP}^2$, and if
$M$ has vanishing second Betti number, then $M$ is diffeomorphic to $S^4$ or $S^1 \times S^3$.}
\end{remark}

\begin{remark}{Taut embeddings of homogeneous spaces}
\label{rem:taut-homog-spaces}

\noindent
{\rm Many important examples of taut embeddings are homogeneous
spaces, e.g., principal orbits of isotropy representations of symmetric spaces.  Thorbergsson \cite{Th5}
found some necessary topological conditions for the existence of a taut
embedding which enabled him to prove that certain homogeneous
spaces do not admit taut embeddings.  Similarly, Hebda
\cite{Heb3} found certain necessary cohomological conditions
for the existence of a taut embedding, and he used
these results to give examples of manifolds which cannot be tautly
embedded. In the case where $M$ is a compact homogeneous
submanifold substantially embedded in Euclidean space
with flat normal bundle, 
Olmos \cite{Olm1} showed that the following statements are
equivalent:} 
\begin{enumerate}
\item[${\rm(a)}$] $M$ is taut; 
\item[${\rm(b)}$] $M$ is Dupin;
\item[${\rm(c)}$] $M$ has constant principal curvatures;
\item[${\rm(d)}$] $M$ is an orbit of the isotropy representation of a symmetric space;
\item[${\rm(e)}$] the first normal space of $M$ coincides with the normal space.
\end{enumerate}
\end{remark}

\begin{remark}{Taut representations}
\label{rem:taut-representations}

\noindent
{\rm Gorodski and Thorbergsson \cite{Gorodski-Thor-02}--\cite{Gorodski-Thor-03}  studied {\em taut representations}, 
i.e., representations of compact Lie groups all of whose orbits are tautly embedded. Bott and Samelson \cite{BS} proved that isotropy representations of symmetric spaces (also called $s$-representations) are taut.  For a long time, the $s$-representations were the only known examples of taut representations,
but in the paper \cite{Gorodski-Thor-03}, Gorodski and Thorbergsson classified taut irreducible representations of compact Lie groups.  Their classification includes three 
families of representations that are not $s$-representations, thereby supplying many new examples of tautly embedded homogeneous spaces.  In a subsequent paper, Gorodski \cite{Gorodski-08}
gave a complete classification of all taut representations of compact simple Lie groups. 

In related work, the class of polar representations was introduced by Dadok and Kac \cite{Dadok-Kac} in 1985. 
In that same year, Dadok \cite{Dadok} proved that a polar representation of a compact Lie group 
has the same orbits as the isotropy representation of a Riemannian symmetric space.   More recently, 
Geatti and Gorodski \cite{Geatti} extended this theory
by showing that a polar orthogonal representation of a connected real reductive
algebraic group has the same closed orbits as the isotropy representation of a semi-Riemannian symmetric space.

In a related area, a proper isometric action of a Lie group $G$ on a Riemannian manifold $M$ is called {\em polar}  if there exists a connected, complete submanifold $\Sigma$ (called a {\em section}) that meets all orbits of $G$
orthogonally.  A basic result is that a section $\Sigma$ is a totally geodesic submanifold of $M$.  Biliotti and Gorodski \cite{Biliotti-Gorodski-07}
proved that the orbits 
of a polar action of a compact Lie group on a compact rank one
symmetric space are ${\bf Z}_2$-taut.

In a recent preprint, Dominguez-Vazquez, Mendes, and Radeschi  \cite{DV4}
classified polar foliations on the Cayley projective plane, thereby completing the classification
of polar foliations on compact rank one symmetric spaces.}
\end{remark}

\begin{remark}{Cylindrically taut immersions}
\label{rem:cylindrically-taut}

\noindent
{\rm Carter, Mansour and West \cite{CMW}, \cite{CW5} introduced a notion of
$k$-{\em cylindrical taut} immersion  $f:M \rightarrow {\bf R}^n$
by using distance functions
from $k$-planes in  ${\bf R}^n$ (see also Carter and \c{S}ent\"{u}rk \cite{CS}, and Carter and West \cite{CW5}).
For $k = 0$, this is equivalent
to tautness, and for $k = n-1$ it is equivalent to tightness.
This theory turns out to be closely related to the theory of convex
sets and many of the results concern embeddings of spheres.
(See also Wegner \cite{Weg} for more on cylindrical distance
functions.)}
\end{remark}

\bigskip
\noindent Department of Mathematics and Computer Science\\
\noindent College of the Holy Cross\\
\noindent Worcester, Massachusetts 01610, U.S.A.\\

\noindent e-mail: tcecil@holycross.edu

\end{document}